\documentclass[11pt]{article}

\usepackage{amsmath, amsfonts, amssymb}
\usepackage{graphicx, amsthm,color}
\usepackage[shortlabels]{enumitem}
\usepackage[dvipsnames]{xcolor}
\usepackage{hyperref}
\usepackage{cleveref}
\usepackage[numbers,sort&compress]{natbib}

\newtheorem{theorem}{Theorem}[section]
\newtheorem{lemma}[theorem]{Lemma}

\newtheorem{corollary}[theorem]{Corollary}
\newtheorem{proposition}[theorem]{Proposition}

\newtheorem{conjecture}{Conjecture}
\theoremstyle{definition}

\theoremstyle{remark}
\newtheorem{remark}[theorem]{Remark}
\numberwithin{equation}{section}

\newcommand{\R}{\mathbb{R}}

\newcommand{\N}{\mathbb{N}}

\newcommand{\Z}{\mathbb{Z}}

\newcommand{\centroid}{\mathfrak{c}}

\providecommand{\vol}{\mathop{\mathcal{H}}\nolimits}
\providecommand{\cone}{\mathop{\rm cone}\nolimits}
\providecommand{\conv}{\mathop{\rm conv}\nolimits}
\providecommand{\proj}{\mathop{\rm proj}\nolimits}

\usepackage{tikz}
\usepackage{tikz-3dplot}
\usetikzlibrary{calc,patterns,angles,quotes}
\usetikzlibrary{babel}
\usepackage{pgfplots}
\usepgfplotslibrary{patchplots}
\usetikzlibrary{backgrounds, intersections}
\usepgfplotslibrary{fillbetween}
\usetikzlibrary{arrows,automata}
\usepackage{subcaption}

\pgfplotsset{compat=1.18}
\usepgfplotslibrary{colorbrewer}

\begin{document}

\title{Reducing the Large Set Threshold for Oertel's Conjecture on the Mixed-Integer Volume\thanks{A short version of this article appeared in the proceedings of IPCO 2025 \cite{CirstiSalas-IPCO2025}. This extended version
contains more detailed discussions, examples, results, and proofs.}}

\author{Andrés Cristi
\thanks{EPFL, Lausanne, Switzerland. {\tt andres.cristi@epfl.ch}}
	\and 			
	David Salas \thanks{Instituto de ciencias de la Ingenier\'ia, Universidad de O'Higgins, Rancagua, Chile. {\tt david.salas@uoh.cl}}
}
\date{\vspace{-1em}}

\maketitle
\thispagestyle{empty}
\begin{abstract}
    In 1960,  B. Gr\"{u}nbaum proved that, for any convex body $C\subset\R^d$ and every halfspace $H$ containing the centroid of $C$, the volume of $H\cap C$ is at least a $\frac{1}{e}$-fraction of the volume of $C$. 
    In  
    2014,
    Oertel conjectured
    that a similar result holds for mixed-integer convex sets. Concretely, he proposed that for any convex body $C\subset \R^{n+d}$, there should exist a point $\mathbf{x} \in S=C\cap(\Z^{n}\times\R^d)$ such that every halfspace $H$ containing $\mathbf{x}$ satisfies
    \[
    \vol_d(H\cap S) \geq \frac{1}{2^n}\frac{1}{e}\vol_d(S),
    \]
    where $\mathcal{H}_d$ denotes the $d$-dimensional Hausdorff measure. While the conjecture remains open, Basu and Oertel proved in 2017 that the above inequality holds for sets that are sufficiently large in terms of a measure known as the \emph{lattice width}. In this work, we improve upon this result, substantially reducing the threshold at which a set can be considered large. We reduce this threshold from an exponential to a polynomial dependency on the dimension, thereby significantly enlarging the family of mixed-integer convex sets for which Oertel's conjecture holds.
\end{abstract}
\thispagestyle{empty}

%%====================================%%
%%====================================%%

\section{Introduction}\label{sec:Intro}

A classic result by Gr\"{u}nbaum~\cite{Grunbaum1960Partitions} states that if $C\subset \R^d$ is a convex body, then every halfspace $H$ of $\R^d$ that contains the centroid of $C$ satisfies
\begin{equation}\label{eq:GrunbaumInequality}
%H\text{ contains the centroid of }C\implies
\frac{\mathrm{vol}_d(H\cap C)}{\mathrm{vol}_d(C)} \geq \left( \frac{d}{d+1}\right)^d\geq \frac{1}{e},
\end{equation}
where $\mathrm{vol}_d(\cdot)$ denotes for the usual $d$-dimensional volume, i.e., the Lebesgue measure over $\R^d.$ 
This result, known as Gr\"{u}nbaum's inequality, together with the study of general measures over convex sets \cite{Grunbaum1960Partitions,Grunbaum1963Measures}, is at the core of many advances in convex geometry (see, e.g., \cite{Klartag2007ACentral,Stephen2017Grunbaum,Bobkov2010Probability,Boroczky2017Cone-volume,Brzezinski2013Volume,Belloni2008OnTheSymmetry,Lovasz2007TheGeometry,Stephen2019Applications}), and has proven to be useful in convex optimization and in the study of cutting plane algorithms \cite{Nemirovsky1983Complexity} and randomized algorithms \cite{Bertsimas2004Solving,He2008Semidefinite,Dabbene2010ARandomized}. A particularly interesting application is the study of the information complexity of convex optimization \cite{BasuOertel2017Centerpoints,Basu2023Complexity,Basu2024Information}: when optimizing a linear objective function over a convex body accessible only via a separation oracle, Gr\"unbaum's inequality gives a bound on the minimum number of oracle queries required to find a solution.

Motivated by this application to convex optimization, 
Oertel~\cite{oertel2014integer}
studied how to generalize Gr\"{u}nbaum's inequality to \textit{mixed-integer convex bodies}, that is, sets of the form $S:=C\cap(\Z^n\times\R^d)$, where $C$ is a convex body. His analysis followed three steps. First, the relative volume of a set $A\subset S$ should be measured by\footnote{This idea of volume for mixed-integer convex sets can be formalized using Hausdorff measures (see Section \ref{sec:Preliminaries}).}
    \begin{equation}\label{eq:MixedInteger-Volume}
    \mu(A) = \frac{\sum_{z\in \Z^n} \mathrm{vol}_d\left( A\cap(\{z\}\times\R^d) \right)}{\sum_{z\in \Z^n} \mathrm{vol}_d\left( S\cap(\{z\}\times\R^d) \right) }.
    \end{equation}
    Second, since the centroid of $S$ can be outside $S$, he proposed to replace it by a point in $S$ that maximizes the halfspace depth with respect to $\mu$, i.e., a point $\mathbf{x}$ that maximizes
    \[
    \inf\{ \mu(H\cap S)\, :\, H\text{ is a halfspace with }\mathbf{x}\in H  \},
    \]
    which he called a \textit{centerpoint}. We denote by $\mathcal{F}(S)$ the maximum halfspace depth normalized by $\mu(S)$, and call it the Oertel radius of $S$ (see Section~\ref{subsec:OertelRadius} below). In other words, a halfspace containing a centerpoint of $S$ is guaranteed to contain a fraction $\mathcal{F}(S)$ of the mixed-integer volume of $S$, and by definition, no better guarantee is possible. 
    %a point $\bar{\mathbf{x}}\in S$ is a centerpoint of $S$ if it is one of the points that attains the maximum in
    %\begin{equation}\label{eq:BasuOertelRadius-general}
    %    \mathcal{F}(S) := \max_{\mathbf{x}\in S}\big(\inf\{ \mu(H\cap S)\, :\, H\text{ is a halfspace with }\mathbf{x}\in H  \} \big).
    %\end{equation}
    %We call $\mathcal{F}(S)$ the Oertel radius of $S$. 
    %Note that $S$ always admits centerpoints, which explains the $\max$ operation in the definition of $\mathcal{F}(S)$ (see \cite{BasuOertel2017Centerpoints}).
    Finally, he studied lower bounds for $\mathcal{F}(S)$, in the spirit of Gr\"{u}nbaum's inequality \eqref{eq:GrunbaumInequality}.
    The centerpoint approach was then extended by Basu and Oertel in \cite{BasuOertel2017Centerpoints}, and used in a series of papers \cite{BasuOertel2017Centerpoints,Basu2023Complexity,Basu2024Information} to study information complexity in convex mixed-integer optimization.
    
\paragraph{Oertel's conjecture on the mixed-integer volume and a matching upper bound:}
Oertel conjectured that the following type of instance should be the worst case for $\mathcal{F}(S)$.  Let $K\subset\R^d$ be a $d$-dimensional simplex.
%\footnote{Throughout this paper, we use the term ($d$-dimensional) \emph{cone} to denote the convex hull of the union of a $(d-1)$-dimensional convex set, the base, and a point $x$ not contained in the affine hull of the base, which we call the apex or vertex of the cone. See \Cref{sec:cone} for a more detailed definition and some properties.} 
Take the mixed-integer convex body $S$ defined by  $C = [0,1]^n\times K$,  i.e., $S = C\cap(\Z^n\times\R^d) = \{0,1\}^n\times K$. In this example, one can easily prove the following facts:
\begin{enumerate}[(i)]
    \item All centerpoints are given by $(z,c)$ with $z\in \{0,1\}^{n}$ and $c$ the centroid of $K$.
    \item Every halfspace $H$ containing a centerpoint $(z,c)$ satisfies
    \[
    \mu(H\cap S) \geq \frac{1}{2^n}\left(\frac{d}{d+1}\right)^d
    \]
    \item For each centerpoint $(z,c)$, there is a vector $u_n\in\R^n$  such that the halfspace in $\R^n$ given by  $H_1 = [u_n^{\top}(w-z) \geq 0]$ strictly separates $z$ from the other points in $\{0,1\}^n$, and there is a vector $u_d\in\R^d$ such that $H_2 = [u_d^{\top}(x-c) \geq 0]$ is the halfspace of $\R^d$ that passes through the centroid of $K$ parallel to one of its bases, and contains exactly a fraction $(d/(d+1))^d$ of the volume of $K$ (tight for Gr\"unbaum's inequality). Then, for a sufficiently large $R>0$, the halfspace $H = [ R u_n^{\top}(w-z) + u_d^{\top}(x-c) \geq 0 ] \subset\R^{n+d}$ verifies
\[
\mu(H\cap S) = \frac{\mathrm{vol}_d(H_2\cap K)}{2^n\mathrm{vol}_d(K)} = \frac{1}{2^{n}}\left(\frac{d}{d+1}\right)^d.
\]
\end{enumerate}
 
Thus, in this example, $\mathcal{F}(S) = \frac{1}{2^n}\left(\frac{d}{d+1}\right)^d$. This combines in one instance the worst cases for both the discrete ($d=0$) and the continuous ($n=0$) setting: since one can isolate the integer variables of $\{0,1\}^n$ and the volume of all the fibers is the same, one is essentially forced to apply Gr\"{u}nbaum's inequality to a single fiber that contains a fraction $\frac{1}{2^n}$ of the total volume. \emph{Oertel's conjecture} is that this is, in fact, the worst scenario.

\begin{conjecture}[{Oertel's conjecture \cite[Conjecture 4.1.20]{oertel2014integer}}]\label{conj:Oertel} Let $C\subset\R^{n+d}$ be a convex body and $S =C\cap (\Z^n\times\R^d)$. Then, 
\[
\mathcal{F}(S) \geq \frac{1}{2^n}\left(\frac{d}{d+1}\right)^d \geq \frac{1}{2^n}\frac{1}{e}.
\]
\end{conjecture}

\paragraph{Partial results on Oertel's conjecture:} Two results stand out in prior work towards settling \Cref{conj:Oertel}. The first, by Oertel himself (see \cite[Theorem 4.1.19]{oertel2014integer} or \cite[Corollary 3.4]{BasuOertel2017Centerpoints}), states that 
for every mixed-integer convex body $S\subseteq \Z^n\times \R^d$,
\begin{equation}\label{eq:BasuOertel-bound}
    \mathcal{F}(S)\geq \frac{1}{2^{n}(d+1)}.
\end{equation}
This was proved with a technique based on Helly numbers (see, e.g., \cite{Basu2023Complexity} and the references therein). While the factor $1/2^n$ is expected after the previous discussion, the factor $1/(d+1)$ matches the conjecture only for the case $d=1$.

The second result, by Basu and Oertel~\cite{BasuOertel2017Centerpoints},\footnote{The works \cite{oertel2014integer,BasuOertel2017Centerpoints} and the subsequent works \cite{Basu2023Complexity,Basu2024Information} are much richer than what we describe here. To streamline the exposition, we focus only on the elements relevant to \Cref{conj:Oertel}.} states that if the set of integer points belonging to the projection of $S$ onto $\R^n$ is sufficiently large, then \Cref{conj:Oertel} holds true. The concept of ``large set'' there is measured in terms of the \textit{lattice width}, which for a set $D\subset\R^n$ is defined as
\begin{equation}
  \omega(D) = \min_{u\in\Z^{n}\setminus\{0\}}\left( \max_{z\in D}u^{\top}z - \min_{z\in D}u^{\top}z \right).  
\end{equation}
At a high level, $\omega(D)$ is the minimum width of the set along integral directions. 
For a set $C$ in $\R^{n+d}$, we denote by $\proj_{\R^n}(C)$ the (orthogonal) projection of $C$ onto $\R^n$, under the natural identification $\R^{n+d} =\R^n\times\R^d$. With this notation, the positive result of Basu and Oertel can be summarized as follows.
\begin{theorem}[{\cite[Theorem 3.6]{BasuOertel2017Centerpoints}}]\label{thm:BasuOertel-LargeSets} There exists a universal constant $\alpha>0$ such that for every $n,d\in\N$ and every convex body $C\subset\R^{n+d}$ with $\omega(\proj_{\R^n}(C)) > 2cn(n+d)^{5/2}\alpha n^{n+1}$ for some $c\in\R_+$, it holds
\[
\mathcal{F}(C\cap (\Z^n\times \R^d)) \geq e^{-\frac{1}{c}-1} +e^{-\frac{2}{c}} - 1.
\]
In particular, if $c\in\R_+$ is such that  $e^{-\frac{1}{c}-1} +e^{-\frac{2}{c}} - 1\geq 2^{-(n+1)}$, then Conjecture \ref{conj:Oertel} holds for $C$. 
\end{theorem}

To get an idea of how large $\omega(\proj_{\R^n}(C))$ needs to be in the above theorem, notice that the expression on the right-hand side of the inequality is increasing in $c$, tends to $1/e$ as $c$ grows, and crosses zero at $c\approx 5.47$. We conclude that the threshold given in Theorem~\ref{thm:BasuOertel-LargeSets} is quite large: for the bound to be meaningful, it requires the lattice width $\omega(\proj_{\R^n}(C))$ to be at least $\Omega( (n+d)^{5/2}n^{n+2} )$.

\paragraph{Our contributions:} %\Cref{conj:Oertel} is still unresolved and the results of \cite{oertel2014integer} and \cite{BasuOertel2017Centerpoints} described above are the state-of-the-art for it. In our work, 
While \Cref{conj:Oertel}  remains unresolved, we focus on reducing the threshold at which a set is large enough so that we can show it holds. Our main results are: 
\begin{itemize}
    \item We prove that there exists a universal constant $\alpha>0$ such that if $\proj_{\R^n}(C)$ contains a ball of radius $k\geq \alpha d^2n^{3/2}$ (or a unimodular copy of it), then \Cref{conj:Oertel} holds (see \Cref{thm:general_n} and \Cref{rem:unimodular-ball} below).
    \item By applying a unimodular transformation, we derive a new threshold in terms of the lattice width of the set. We prove that there exists a universal constant $\alpha'>0$ such that if $\omega(\proj_{\R^n}(C))\geq \alpha' d^2n^5$, then \Cref{conj:Oertel} holds for $C$ (see \Cref{cor:Main-LatticeWidth} below).
\end{itemize}

These results improve upon \Cref{thm:BasuOertel-LargeSets} by reducing the exponential dependence on $n$ to an explicit polynomial one, and the dependence on $d$ from $d^{5/2}$ to $d^2$. This considerably enlarges the family of sets for which \Cref{conj:Oertel} holds, and provides further indications that it might hold in general.

As a warm-up, we consider the special case of $n=1$, for which the proofs are simpler. For this case, we also obtain an improved bound: there is a constant $\alpha''>0$ such that $k\geq \alpha'' d$ suffices for \Cref{conj:Oertel} to hold.

The rest of the paper is organized as follows. In \Cref{sec:Preliminaries}, we provide preliminaries and fix some important notation. In \Cref{sec:case-n1}, we present a special proof for the case $n=1$, and in \Cref{sec:general}, we present the general case and our main results. Both  \Cref{sec:case-n1} and \Cref{sec:general} are organized in the same way: we first provide an outline of the proof, then we present all the elements as technical lemmas, and finally, we provide the main theorem of the section.

%%====================================%%
%%====================================%%

\section{Preliminaries}\label{sec:Preliminaries}

In what follows, let $n,d\in \N$ and let $C\subset\R^{n+d}$ be a convex body, that is, a compact convex set with nonempty interior. We call $S = C\cap (\Z^n\times \R^d)$ the mixed-integer set induced by $C$. We denote by $B_n(x,r)$ the closed Euclidean ball in $\R^n$ centered at $x$ with radius $r$. If there is no ambiguity, we will simply write $B(x,r)$.

It will be useful to adopt the following convention: We identify $\R^{n+d}$ with $\R^n\times\R^d$.  We use the letters $z,w$ to denote elements of $\R^n$ and letters $x,y$ to denote elements of $\R^d$. We use bold characters $\mathbf{x},\mathbf{y}$ to denote elements of $\R^n\times\R^d$. Then, for every $\mathbf{x}\in\R^{n+d}$, we write $\mathbf{x} = (z,x) \in\R^n\times\R^{d}$. 

For a set $A\subset\R^p$, we denote by $\mathrm{conv}(A)$ and by $\mathrm{aff}(A)$ the convex and affine hull respectively. We also write $\mathrm{dist}(\cdot,A)$ to denote the distance function to $A$. If $A$ is convex and closed, we also write $\proj_{A}$ to denote its metric projection. 

We formalize the notion of volume for general sets using Hausdorff measures. For every $r\geq 0$, we denote by $\vol_r$ the $r$-dimensional Hausdorff measure (see, e.g., \cite[Definition 2.1]{Evans2015Measure}). When $r$ is an integer, $\vol_r$ coincides with the $r$-dimensional Lebesgue measure over each affine space of dimension $r$. For any set $K\subset \R^{n+d}$ we denote by $\dim K$ its Hausdorff dimension (see, e.g., \cite[Definition 2.2]{Evans2015Measure}) which is given as
\[
\dim K = \inf\{ r>0\ :\ \vol_r(K) = 0\}.
\]
When $K$ is a convex set, $\dim K$ coincides with the dimension of the affine hull of $K$, $\mathrm{aff}(K)$. Thus, if $K$ is a convex body and $\dim K = q$, then the usual $q$-dimensional volume of $K$ coincides $\mathcal{H}_q(K)$. When $K$ is a disjoint union of convex sets (as is the case for mixed-integer convex sets), the Hausdorff dimension of $K$ is the maximum Hausdorff dimension of its convex components. In particular, if the convex body $C\subset{\R^{n+d}}$ has at least one interior point $\mathbf{x} = (z,x)$ with $z\in\Z^n$, then $\dim S = d$. Thus, in this case, the measure $\mu(\cdot)$ defined in \Cref{eq:MixedInteger-Volume} coincides with the normalization of $\mathcal{H}_d$, that is,
\[
\forall A\subset S,\quad \mu(A) = \frac{\sum_{z\in\Z^n} \mathrm{vol}_d(A\cap (\{z\}\times \R^d))}{\sum_{z\in\Z^n} \mathrm{vol}_d(S\cap (\{z\}\times \R^d))} = \frac{\vol_d(A)}{\vol_d(S)}.
\]
In general, for every compact set $K$,
\[
\vol_q(K) = \begin{cases}
    +\infty\quad&\text{ if }q <\dim K,\\
    V\in [0,+\infty)\quad&\text{ if }q =\dim K,\\
    0\quad&\text{ if }q >\dim K.
\end{cases}
\]
Thus, the Hausdorff measure $\mathcal{H}^q$ provides a good extension of the $q$-dimensional volume of an arbitrary compact set. Moreover, the ``intrinsic'' volume of a set $K$ is given by $\mathcal{H}_{\dim K}(K)$. For further information on Hausdorff measures, we refer the reader to \cite{Evans2015Measure}.

For a convex set $K$ of dimension $q = \dim K$, we define the centroid of $K$
\begin{equation}
    \centroid(K) = \frac{1}{\vol_q(K)}\int_{K} x \;d\vol_q(x).
\end{equation}
Note that the centroid is a vector in $\mathrm{aff}(K)$ and, since $K$ is convex, it is contained in $K$. 

\subsection{Some facts about (truncated) cones}
\label{sec:cone}

In this work, we use a somewhat unconventional definition of a cone that is sometimes used in convex geometry (see, e.g., \cite{Ball1997:AnElementary}, or the original work of Gr\"{u}nbaum \cite{Grunbaum1960Partitions}), and is very useful for our analysis. We call a set $K\subset \R^p$ (with $p\geq 2$) a \emph{cone} if it is the convex hull of a point $x\in \R^p$ and a compact convex set $B$ such that $x\notin\mathrm{aff}(B)$. This is a generalization of the notion of a three-dimensional geometric cone, also called a pyramid when the base is a polytope. We insist it does not correspond to the usual definition of a cone used in linear algebra and convex analysis. 

To ease the notation and to emphasize the base and apex of the cone, we write $K = \cone(x,B)$, and define its height as $h = \mathrm{dist}(x,\mathrm{aff}(B))$. We recall the following facts about cones, which will be used in the sequel.

First, the volume of a cone $K=\cone(x,B)$ with $q=\dim B$ and $q+1=\dim K$ is given by
    \begin{equation}
        \vol_{q+1}(K) = \frac{1}{q+1}\cdot h\cdot \vol_q(B),
    \end{equation}
    where $h$ is the height of $K$.

Now, let $K' = \cone(x,B')$ be a subcone of $K$ with height $h' < h$, that is, $\mathrm{aff}(B')$ is parallel to $\mathrm{aff}(B)$ with $B' = K\cap \mathrm{aff}(B')$ and $h' = \mathrm{dist}(x,\mathrm{aff}(B'))$; see Figure \ref{fig:Cone-and-Subcone}. In this case, the volumes of $K$ and $K'$ satisfy
    \begin{equation}
        \vol_{q+1}(K') = \left(\frac{h'}{h}\right)^{q+1}\vol_{q+1}(K).
    \end{equation}

\begin{figure}[ht]
    \centering
    
\begin{tikzpicture}

% Parameters
\def\radius{2}    % radius cone
\def\height{4}    % height cone
\def\cutheight{1.5} % height subcone
\def\cutradius{2.5/4*\radius} % radius subcone

% Base cone C
\draw[dashed,fill=gray!20!white, opacity=0.5] (-\radius, 0) arc[start angle=180, end angle=0, x radius=\radius, y radius=0.5*\radius];
\draw[fill=gray!20!white, opacity=0.5] (-\radius, 0) arc[start angle=180, end angle=360, x radius=\radius, y radius=0.5*\radius];

% Cone
\draw (-2, \height) -- (-\radius, 0);
\draw (-2, \height) -- (\radius, 0);

% Base subcone
\draw[dashed,fill=gray!20!white, opacity=0.5] (-0.75-\cutradius, \cutheight) arc[start angle=180, end angle=0, x radius=\cutradius, y radius=0.4*\cutradius];
\draw[fill=gray!20!white, opacity=0.5] (-0.75-\cutradius, \cutheight) arc[start angle=180, end angle=360, x radius=\cutradius, y radius=0.4*\cutradius];
 
% Labels
\draw[<->] (-3, \cutheight) -- (-3, 0);
\draw[<->] (-3, \height) -- (-3, \cutheight) node[midway, right] {$h'$};
\node at (0, 0.3) {$B$};
\node at (-0.75, 1.6) {$B'$};
\draw [thick, blue,decorate,decoration={brace,amplitude=2pt,mirror},xshift=-2pt,yshift=-0.4pt](-3.1,\height) -- (-3.1,0) node[black,midway,xshift=-0.4cm] {$h$};
\draw [thick, blue,decorate,decoration={brace,amplitude=2pt,mirror},xshift=2pt,yshift=-0.4pt](-0.75+\cutradius,\cutheight) -- (-2,\height) node[black,midway,xshift=0.4cm, yshift=0.2cm] {$K'$};
\end{tikzpicture}
    \caption{Subcone $K' = \cone(x,B')$ of $K = \cone(x,B)$ with height $h'<h$.}
    \label{fig:Cone-and-Subcone}
\end{figure}
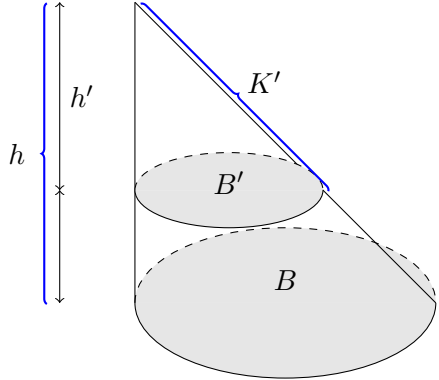

Finally, let $V$ be the (unique) affine space parallel to the base $B$ containing the centroid of $K$. Then, the height $h'$ of the subcone $K' = \cone(x, V\cap K)$ is proportional to the height of $K$, verifying that $h' = \frac{\dim K}{1+\dim K}h$. In other words, the centroid of $K$ divides the height in the proportion $1:\dim K$.

\subsection{Oertel radius}\label{subsec:OertelRadius}

Let $K\subseteq\R^p$. For $x\in K$ and $u \in \R^p\setminus\{0\}$ we define $H(u,x) = \{ y\in\R^p\,:\, u^{\top}(y-x)\geq 0 \}$. Let $q=\dim K$. We can now rewrite the definition of \emph{Oertel radius} of $K$ as
\begin{equation}\label{eq:OertelRadius}
\mathcal{F}(K) = \sup_{x\in K}\inf_{u\in \R^p\setminus\{0\}} \frac{\vol_{q}(H(u,x)\cap K)}{\vol_q(K)}.
\end{equation}

Note that we consider the $(\dim K)$-Hausdorff measure of $H(u,x)\cap K$ rather than its volume in its own dimension, to put zero measure on the extreme cases where $H(u,x)\cap K$ has lower dimension. Also note that when $K$ is a mixed-integer convex body, the Oertel radius is attained by some elements of $K$, as we discussed in the introduction (see \cite{BasuOertel2017Centerpoints}).

%%====================================%%
%%====================================%%

\section{The case $n=1$}\label{sec:case-n1}

Throughout this section, we consider $n=1$, that is, $C\subseteq \R^{d+1}$ and $S=C\cap (\Z\times \R^d)$. Without loss of generality, we assume that $\proj_{\R}(C) = [0,k]$, by replacing $C$ with $\conv(S)$ and moving the origin if necessary. We denote the connected components of $S$ by $S_0,\dots,S_k$, ordered by increasing integral coordinate. Finally, for $i\in\{0,\ldots,k-1\}$, we denote by $C_i$ the section of width $1$ of $C$ defined as $C_i= C\cap \{ (z,x)\in \R^{d+1}: z\in [i,i+1]\}$. By convention, we set $C_k=\emptyset$.

Our goal is to prove that there exists a universal constant $\alpha>0$ such that
\begin{equation}
k\geq \alpha d \implies \mathcal{F}(S) \geq \frac{1}{2}\left(\frac{d}{d+1}\right)^d,
\end{equation}
see Theorem \ref{thm:Main-OneDimension} below. To do so, our strategy can be split in the following steps:
\begin{enumerate}
    \item First, we prove that the volume of each section $C_i$ is small with respect to the volume of $C$ when $k$ is large. This is Lemma \ref{lem:single_face_bound}.

    \item We show that $\vol_{d}(S)$ is well approximated by $\vol_{d+1}(C)$ when $k$ is large. We do so by building inner and outer approximations of each $C_i$ using cones built from either $S_i$ or $S_{i+1}$ and projecting to extreme points, either in $S_0$ or in $S_k$.  This is \Cref{lem:upper_and_lower_bound_n_1}.

    \item Finally, we show that, by losing only a small fraction of the volume, we can move the centroid of $C$ so that it belongs to $S$. This is Lemma \ref{lem:smaller_set_integral_centroid}.
\end{enumerate}

By carefully balancing all these bounds and the associated errors, we proceed as follows: we take $\bar{\mathbf{x}}$ to be the moved centroid that has integer projection (and belongs to the interior of $C$). We then apply Gr\"{u}nbaum's Inequality \eqref{eq:GrunbaumInequality} and the lemmas to show that, for any halfspace $H$ containing the point $\bar{\mathbf{x}}$,
\[
 \frac{\vol_d(H\cap S)}{\vol_d(S)}\geq \left(1-O(d/k)\right)\frac{\vol_{d+1}(H\cap C)}{\vol_{d+1}(C)} \geq \frac{1}{e} -O(d/k).
\]
Thus, leveraging the inequality $\frac{1}{e} > \frac{1}{4} \geq \frac{1}{2}\left(\frac{d}{d+1}\right)^d$, we conclude that it is enough that $k\geq \alpha d$ for some universal $\alpha$ to conclude the bound of the conjecture.

\subsection{Technical Lemmas for the case $n=1$}

In this section, we present the three technical lemmas described in the outline of the proof. Throughout this section, we consider $k\gg d$ so that $1-O(d/k)$ is positive, $O(d/k)^2 = O(d/k)$ and $(1+O(d/k))^{-1} = 1-O(d/k)$.
%%------------------------------------%%
%%------------------------------------%%

\begin{lemma}\label{lem:single_face_bound}
    For every $C_i$, we have that
    \[
    \vol_{d+1}(C_i)\leq O\left( \frac{d}{k} \right)  \cdot \vol_{d+1}(C).
    \]
\end{lemma}

\begin{proof}
    Assume without loss of generality that $(0,0)$ and $(k,0)$ belong to $C$. Assume $S_i$ is at a distance of at least $k/2$ from $(0,0)$, i.e., $i\geq k/2$ (otherwise, apply the analogous analysis from the other endpoint $(k,0)$). Since $(0,0)$ and $S_i$ are both in $C$, their convex hull $L=\cone((0,0),S_i)$ is contained in $C$. Also, we have that $C_i\subseteq ((i+1)/i)\cdot L$, as the line connecting a point in $C_i$ and $(0,0)$ is in $C$, so it must intersect $S_i$. See Figure \ref{fig:Ci-in-enlargementL}.
    
    \begin{figure}[ht]
        \centering
        \begin{tikzpicture}
    % Dibuja el conjunto convexo no simétrico con bordes lisos
    \filldraw[thick, fill=gray!70!white, fill opacity=0.1]    (0,0) 
        .. controls (1.5,1.5) and (4,2) .. (6,0.5)  % Borde superior
        -- (6,-0.5) % Extremo derecho inferior
        .. controls (2,-3) and (0.5,-1) .. (0,0); % Borde inferior

    % Líneas rojas punteadas
    \draw[dashed] (2,-4) -- (2,3);
    %\draw[dashed] (4,-4) -- (4,3);
   \draw[dashed] (3.3,-4) -- (3.3,3);
    % Define los puntos de intersección
    \path [name path=lineA] (2,-4) -- (2,3);  % Línea vertical en x=2
    \path [name path=lineB] (3.3,-4) -- (3.3,3);  % Línea vertical en x=4
    \path [name path=upperCurve] (0,0) .. controls (1.5,1.5) and (4,2) .. (6,0.5);
    \path [name path=lowerCurve] (6,-0.5) .. controls (2,-3) and (0.5,-1) .. (0,0);

    % Intersección de la línea en x=2 con las curvas
    \path [name intersections={of=lineA and upperCurve, by=upperA}];
    \path [name intersections={of=lineA and lowerCurve, by=lowerA}];
    
    % Intersección de la línea en x=4 con las curvas
    \path [name intersections={of=lineB and upperCurve, by=upperB}];
    \path [name intersections={of=lineB and lowerCurve, by=lowerB}];

    % Mostrar las coordenadas de los puntos
    \coordinate (p1) at (upperA);
    \coordinate (p2) at (lowerA);
    \coordinate (p3) at (upperB);
    \coordinate (p4) at (lowerB);

    \draw[thick,red] (0,0)--(3.28,2);
    \draw[thick,red] (0,0)--(3.28,-2.7);
    \draw[thick,red] (3.28,-2.7)--(3.28,2);
    \node[right,red] at (3.4,2){\scriptsize $((i+1)/i)\cdot L$};

   \draw[very thick,teal] (p3)--(p4);
   \draw [ decorate,decoration={brace,amplitude=4pt,mirror},xshift=0.4pt,yshift=-0.4pt](p4) -- (p3) node[black,midway,xshift=0.6cm] {\scriptsize $S_{i+1}$};
    
    \draw[very thick,teal] (p1)--(p2);
    \draw [ decorate,decoration={brace,amplitude=4pt,mirror},xshift=-1.2pt,yshift=-0.4pt](p1) -- (p2) node[black,midway,xshift=-0.5cm] {\scriptsize $S_{i}$};

    \draw[<->, thick] (0,-3.5)--(6,-3.5);
    \fill[black] (2,-3.5) circle (2pt) node[below left]{\small $i$};
    \fill[black] (3.3,-3.5) circle (2pt) node[below right]{\small $i+1$};
    \node at (2.6,0) {\scriptsize $C_i$};

    \fill[red] (0,0) circle (2pt) node[above left]{\scriptsize $(0,0)$};
\end{tikzpicture}
        \caption{Illustration that $L\subset C$ and $C_i \subset ((i+1)/i)\cdot L$.}
        \label{fig:Ci-in-enlargementL}
    \end{figure}
    
    Thus, the fraction of volume in $C_i$ with respect to $C$ must be at most the fraction of volume in $A:= (((i+1)/i)\cdot L) \setminus L$ with respect to $B:=L$. Using concavity of the function $g(t)=(t^{d+1}-1)/t^{d+1} =  1-1/t^{d+1}$ for $t>0$, which has derivative $g'(t)= (d+1)/t^{d+2}$, and noting $((i+1)/i)^d\leq ((d+1)/d)^d \leq e$ we get that
    \begin{align*}
     \frac{\mathcal{H}_{d+1}(C_i)}{\mathcal{H}_{d+1}(C) }\leq\frac{\mathcal{H}_{d+1}(A)}{\mathcal{H}_{d+1}(B) } &= e\frac{
    \left( 1+1/i\right)^{d+1} -1}{\left( 1+1/i\right)^{d+1}} 
    \leq g'(1)\cdot \frac{e}{i}
    = \frac{e(d+1)}{i} \leq \frac{2e(d+1)}{k}
    = O(d/k).
    \end{align*}
    
\end{proof}

%%------------------------------------%%
%%------------------------------------%%

\begin{lemma}
\label{lem:upper_and_lower_bound_n_1}
    We have that
    \[
    \left(1-O\left(\frac{d}{k}\right)\right) \vol_{d+1} (C) \leq \vol_d(S)\leq \left(1+O\left(\frac{d}{k}\right)\right) \vol_{d+1} (C).
    \]
\end{lemma}

\begin{proof}
    Fix points $\mathbf{x}_0 \in S_0$ and $\mathbf{x}_k \in S_k$. For the upper bound, we create a lower bound on $\vol_{d+1}(C)$ using $\vol_d(S)$. For each $S_i$, the intersection of $\cone(\mathbf{x}_0,S_i)$ with $[i-1,i]\times \R^d$ is contained in $C_{i-1}$. Similarly, the intersection of $\cone(\mathbf{x}_k,S_i)$ with  $[i,i+1]\times \R^d$ is contained in $C_i$. For each $i$ we choose the side that is farthest from the corresponding vertex, so we define
    \[
    \check{K}_i = \begin{cases}
        \cone(\mathbf{x}_0,S_i) \cap ([i-1,i]\times \R^d) \quad&\text{ if }i>k/2,\\
        \cone(\mathbf{x}_k,S_i)\cap  ([i,i+1]\times \R^d )\quad&\text{ if }i\leq k/2.
    \end{cases}
    \]
    These sets provide lower bounds for each $C_i$, with the section $i=\lfloor k/2 \rfloor$ counted twice. See \Cref{fig:InnerApproximation-Cones}. Also, because the heights of the cones are all at least $k/2$, we have  that $\vol_{d+1}(\check{K}_i) \geq (1-O(d/k)) \vol_d(S_i)$, so,
    \begin{align*}
        \left(1-O\left(\frac{d}{k}\right)\right) \cdot \vol_d(S) &\leq  \sum_{i=0}^k \vol_{d+1} (\check{K}_i)\\
        &\leq  \sum_{i=0}^{k-1} \vol_{d+1} (C_i) + \vol_{d+1}\left(C_{\lfloor k/2\rfloor}\right)\leq \vol_{d+1}(C) + O\left(\frac{d}{k}\right)\cdot \vol_{d+1} (C),
    \end{align*}
    where in the last inequality we used \Cref{lem:single_face_bound} for the extra section $C_{\lfloor k/2 \rfloor}$ counted twice. This concludes the proof of the upper bound.

    \begin{figure}[ht]
    \begin{subfigure}[t]{0.5\textwidth}
        \centering
        \begin{tikzpicture}[scale=0.9]
    % Dibuja el conjunto convexo no simétrico con bordes lisos
    \draw[name path=convexo, fill=gray!50!white, fill opacity=0.1] (0,0) 
        .. controls (1.5,1.5) and (4,2) .. (6.5,0.5)  % Borde superior
        -- (6.5,-0.5) % Extremo derecho inferior
        .. controls (2,-3) and (0.5,-1) .. (0,0); % Borde inferior

    % Dibuja líneas verticales punteadas y encuentra intersecciones
    \draw[dashed] (0,-2.5) -- (0,2.5);
    \foreach \x in {1,2,3,4,5,6,7,8,9,10,11,12,13} {
        \draw[name path=line\x, dashed] (0.5*\x,-2.5) -- (0.5*\x,2.5);
        
        % Encuentra las intersecciones de cada línea punteada con el borde del conjunto
        \path[name intersections={of=convexo and line\x, by={A\x, B\x}}];
        
        % Marca las intersecciones con pequeños círculos
        %\fill (A\x) circle (1.5pt);
        %\fill (B\x) circle (1.5pt);
    }
    \foreach \x in {7,8,9,10,11,12} {
        \begin{scope}
		\clip (0.5*\x-0.5,-2.5) -- (0.5*\x-0.5,2.5) -- (0.5*\x,2.5) -- (0.5*\x,-2.5)--cycle; 
		
		\filldraw[thick,blue, pattern=north west lines, pattern color = blue, fill opacity = 0.3] (0,0)--(A\x)--(B\x)--cycle;
		\end{scope}
    }
    \begin{scope}
		\clip (6,-2.5) -- (6,2.5) -- (6.5,2.5) -- (6.5,-2.5)--cycle; 
		
		\filldraw[thick,blue, pattern=north west lines, pattern color = blue, fill opacity = 0.3] (0,0)--(6.5,0.5)--(6.5,-0.5)--cycle;
		\end{scope}
    \foreach \x in {1,2,3,4,5,6} {
        \begin{scope}
		\clip (0.5*\x,-2.5) -- (0.5*\x,2.5) -- (0.5*\x+0.5,2.5) -- (0.5*\x+0.5,-2.5)--cycle; 
		
	\filldraw[thick,red, pattern=north west lines, pattern color = red, fill opacity = 0.3] (6,0)--(A\x)--(B\x)--cycle;
		\end{scope}
    }
    \draw[red,thick] (0,0)--(0.5,0);
    
    \draw[<->, thick] (-0.5,-2.5)--(7,-2.5);
    \foreach \x in {0,1,2,3,4,5,6}{
    \fill[black] (0.5*\x,-2.5) circle (2pt) node[below left]{\tiny \x};
    }
    
    \foreach \x in {7,8,9,10,11,12,13}{
    \fill[black] (0.5*\x,-2.5) circle (2pt) node[below right]{\tiny \x};
    }
    
    \fill[blue] (0,0) circle (2pt) node[below left]{\tiny $x_0$};
    \fill[red] (6.5,0) circle (2pt) node[below right]{\tiny $x_k$};
    
    \draw (3.25,-2.45)--(3.25,-2.55);
    \node[below] at (3.25,-2.55){\tiny$\frac{k}{2}$};
\end{tikzpicture}
        \caption{Inner approximation.}
        \label{fig:InnerApproximation-Cones}
    \end{subfigure}
        \begin{subfigure}[t]{0.5\textwidth}
        \begin{tikzpicture}[scale=0.9]
    % Dibuja el conjunto convexo no simétrico con bordes lisos
    \draw[name path=convexo, fill=gray!50!white, fill opacity=0.1] (0,0) 
        .. controls (1.5,1.5) and (4,2) .. (6.5,0.5)  % Borde superior
        -- (6.5,-0.5) % Extremo derecho inferior
        .. controls (2,-3) and (0.5,-1) .. (0,0); % Borde inferior

    % Dibuja líneas verticales punteadas y encuentra intersecciones
    \draw[dashed] (0,-2.5) -- (0,2.5);
    \foreach \x in {1,2,3,4,5,6,7,8,9,10,11,12,13} {
        \draw[name path=line\x, dashed] (0.5*\x,-2.5) -- (0.5*\x,2.5);
        
        % Encuentra las intersecciones de cada línea punteada con el borde del conjunto
        \path[name intersections={of=convexo and line\x, by={A\x, B\x}}];
        
        % Marca las intersecciones con pequeños círculos
        %\fill (A\x) circle (1.5pt);
        %\fill (B\x) circle (1.5pt);
    }
    \foreach \x in {7,8,9,10,11,12} {
        \begin{scope}
		\clip (0.5*\x,-2.5) -- (0.5*\x,2.5) -- (0.5*\x + 0.5,2.5) -- (0.5*\x +0.5,-2.5)--cycle; 
		
		\filldraw[thick,blue, pattern=north west lines, pattern color = blue, fill opacity = 0.3] (0,0)--($ (0,0)!1.2!(A\x) $)--($ (0,0)!1.2!(B\x) $)--cycle;
		\end{scope}
    }
    \begin{scope}
		\clip (6.5,-2.5) -- (6.5,2.5) -- (7,2.5) -- (7,-2.5)--cycle; 
		
		\filldraw[thick,blue, pattern=north west lines, pattern color = blue, fill opacity = 0.3] (0,0)--($(0,0)!1.2!(6.5,0.5)$)--($(0,0)!1.2!(6.5,-0.5)$)--cycle;
		\end{scope}
    \foreach \x in {1,2,3,4,5,6} {
        \begin{scope}
		\clip (0.5*\x-0.5,-2.5) -- (0.5*\x-0.5,2.5) -- (0.5*\x,2.5) -- (0.5*\x,-2.5)--cycle; 
		
	\filldraw[thick,red, pattern=north west lines, pattern color = red, fill opacity = 0.3] (6,0)--($(6,0)!1.2!(A\x)$)--($(6,0)!1.2!(B\x)$)--cycle;
		\end{scope}
    }
    \draw[red,thick] (-0.5,0)--(0,0);
    
    \draw[<->, thick] (-0.5,-2.5)--(7,-2.5);
    \foreach \x in {0,1,2,3,4,5,6}{
    \fill[black] (0.5*\x,-2.5) circle (2pt) node[below left]{\tiny \x};
    }
    
    \foreach \x in {7,8,9,10,11,12,13}{
    \fill[black] (0.5*\x,-2.5) circle (2pt) node[below right]{\tiny \x};
    }
    
    \fill[blue] (0,0) circle (2pt) node[below left]{\tiny $x_0$};
    \fill[red] (6.5,0) circle (2pt) node[below right]{\tiny $x_k$};
    
    \draw (3.25,-2.45)--(3.25,-2.55);
    \node[below] at (3.25,-2.55){\tiny$\frac{k}{2}$};
\end{tikzpicture}
        \caption{Outer approximation.}        \label{fig:OuterApproximation-Cones}
        \end{subfigure}
        \caption{Illustrations of inner and outer approximations from the proof of \Cref{lem:upper_and_lower_bound_n_1}. The section $C_{\lfloor k/2\rfloor}$ is approximated twice in the inner approximation or is missing in the outer approximation.}
    \end{figure}

    For the lower bound, we do the opposite. For a point $\mathbf{x}$ and a set $A$, define
    \[
    \cone_\infty(\mathbf{x},A) = \{ \mathbf{x} + t\cdot (\mathbf{y} - \mathbf{x}): t\geq 0, \mathbf{y}\in A\}.
    \]
    With this, we can define
        \[
    \hat{K}_i = \begin{cases}
        \cone_\infty(\mathbf{x}_0,S_i) \cap ([i,i+1] \times \R^d ) \quad&\text{ if }i>k/2,\\
        \cone_\infty(\mathbf{x}_k,S_i) \cap ( [i-1,i]\times \R^d ) \quad&\text{ if }i\leq k/2.
    \end{cases}
    \]
    In other words, we define $\hat{K}_i$ by projecting $S_i$ to the opposite side from $\check{K}_i$. By convexity, we have that for $i>k/2$, $C_i\subseteq \hat{K}_i$, and for $i\leq k/2$, $C_{i-1}\subseteq \hat{K}_i$. Thus, we get upper bounds for each $C_i$, except for $i=\lfloor k/2 \rfloor$. See \Cref{fig:OuterApproximation-Cones}. Similar to the upper-bound argument, we get that
    \begin{align*}
        \left(1+O\left(\frac{d}{k}\right)\right)\cdot \vol_d(S) &\geq \sum_{i=0}^k \vol_{d+1} (\hat{K}_i)\\&\geq \vol_{d+1}(C) - O\left(\frac{d}{k}\right) \cdot \vol_{d+1}(C),
    \end{align*}
    where again we used \Cref{lem:single_face_bound} for the missing section. This concludes the proof of the lower bound.
\end{proof}

%%------------------------------------%%
%%------------------------------------%%

\begin{lemma}
    \label{lem:smaller_set_integral_centroid}
    There is a convex body $C'\subseteq C$ with centroid in $\Z\times \R^d$ and such that
    \[
    \vol_{d+1}(C\setminus C') \leq O\left(\frac{d}{k}\right) \cdot \vol_{d+1}(C).
    \]
\end{lemma}

\begin{proof}
    For a set $A$, denote by $z(A)$ the first coordinate of the centroid of $A$. In this proof, we assume that $z(C)=0$ and we show that it is possible to continuously remove part of the set until the new set $C'$ satisfies $z(C')\leq -1$, which results in a loss of at most an $O(k/d)$-fraction of the original volume. This guarantees that, in the general case, the centroid of $C'$ can be placed at the integral coordinate $z(C') = \lfloor z(C)\rfloor \geq z(C)-1$, while losing no more than an $O(k/d)$-fraction of the volume. See \Cref{fig:Moving-centroid}.

            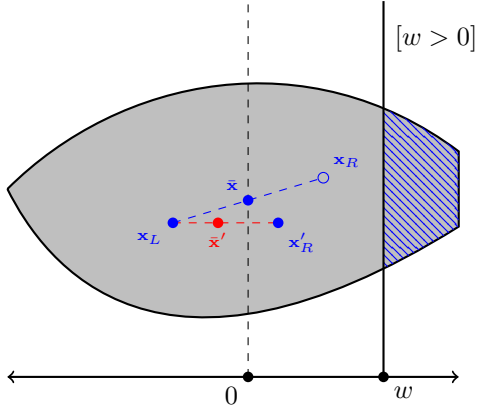
\begin{figure}[ht]
        \centering
        \begin{tikzpicture}
    % Dibuja el conjunto convexo no simétrico con bordes lisos
    \draw[thick, fill=gray!50!white, fill opacity=0.1]    (0,0) 
        .. controls (1.5,1.5) and (4,2) .. (6,0.5)  % Borde superior
        -- (6,-0.5) % Extremo derecho inferior
        .. controls (2,-3) and (0.5,-1) .. (0,0); % Borde inferior
    %\draw[line width=1pt] (0,-2)--(6,1.5);
    \node at (5.7,2) {\small $[w>0]$};
    \begin{scope}
		\clip (0,0) 
        .. controls (1.5,1.5) and (4,2) .. (6,0.5)  % Borde superior
        -- (6,-0.5) % Extremo derecho inferior
        .. controls (2,-3) and (0.5,-1) .. (0,0); 
		
		\fill[pattern=north west lines, pattern color = blue, fill opacity = 0.3] (5,-2) -- (6,-2)--(6,2)--(5,2)--cycle;
		\end{scope}

    \draw[dashed] (3.2,-2.5) -- (3.2,2.5);
    \draw[thick] (5,-2.5) -- (5,2.5);

    \draw[dashed,blue] (2.2,-0.45)--(4.2,0.15);
    \draw[dashed,red] (2.2,-0.45)--(3.6,-0.45);
    
    \fill[blue] (3.2,-0.15) circle (2pt) node[above left] {\tiny $c(C)$};

    \fill[blue] (2.2,-0.45) circle (2pt) node[below left] {\tiny $c(C_L)$};
    \draw[blue] (4.2,0.15) circle (2pt) node[above] {\tiny $c(C_R)$};
    \fill[blue] (3.6,-0.45) circle (2pt) node[below right] {\tiny $c(C^w_R)$};
    \fill[red] (2.8,-0.45) circle (2pt) node[below] {\tiny $c(C^w)$};

     \draw[<->, thick] (0,-2.5)--(6,-2.5);
    \fill[black] (3.2,-2.5) circle (2pt) node[below left]{\small $0$};
    \fill[black] (5,-2.5) circle (2pt) node[below right]{\small $w$};
    \end{tikzpicture}
        \caption{Illustration of the modification of the centroid after cutting a part of $C$. The centroid of $C_R^w$ is to the left of the centroid of $C_R$. Accordingly, the new centroid of $C^w$ is also moved to the left.}
        \label{fig:Moving-centroid}
    \end{figure}
    
    Let $C_L$ and $C_R$ be the partition of $C$ into the sets to the left and to the right of $\{0\}\times \R^d$. Since $z(C)=0$, we have that
    \begin{align}
        z(C_L)\cdot \vol_{d+1}(C_L) + z(C_R)\cdot \vol_{d+1}(C_R)  =0.        \label{eq:centroidequation}
    \end{align}
    For $w\geq 0$, define $C^w = C\cap ((-\infty,w]\times \R^d)$ and $C^w_R = C\cap ([0,w]\times\R^d)$. We have that
    \begin{align*}
        z(C^w) &=  z(C_L) \cdot \frac{\vol_{d+1}(C_L)}{\vol_{d+1}(C^w)} + z(C^w_R)  \cdot \frac{\vol_{d+1}(C^w_R)}{\vol_{d+1}(C^w)}\\
        &\leq z(C_L) \cdot \frac{\vol_{d+1}(C_L)}{\vol_{d+1}(C^w)} + z(C_R)  \cdot \frac{\vol_{d+1}(C^w_R)}{\vol_{d+1}(C^w)}\\
        &= - z(C_R) \cdot \frac{\vol_{d+1}(C_R)}{\vol_{d+1}(C^w)} + z(C_R) \cdot \frac{ \vol_{d+1}(C^w_R)}{\vol_{d+1}(C^w)}\\
        &= - z(C_R) \cdot\frac{\vol_{d+1}(C_R\setminus C^w_R)}{\vol_{d+1}(C^w)}\\
        &= -z(C_R) \cdot \frac{\vol_{d+1}(C\setminus C^w)}{\vol_{d+1}(C^w)} \leq -z(C_R) \cdot \frac{\vol_{d+1}(C\setminus C^w)}{\vol_{d+1}(C)},
    \end{align*}
    where in the third line we used \Cref{eq:centroidequation}. Thus, if we can take $w$ such that
    \[\frac{\vol_{d+1}(C\setminus C^w)}{\vol_{d+1}(C)} \geq \frac{1}{z(C_R)},\]
    then we get that $z(C^w) \leq -1$. By Gr\"unbaum's inequality, $C_R$ contains at least a fraction $1/e$ of the volume of $C$, and by \Cref{lem:single_face_bound}, each $C_i$ contains only an $O(d/k)$ fraction of the volume. Thus, by taking $H= \{(z,x)\mid z\leq z(C_R)\}$ and applying again Gr\"unbaum's inequality to $C_R$, we get that
    \[
    \frac{1}{e^2}\mathcal{H}_{d+1}(C) \leq \mathcal{H}_{d+1}(C_R\cap H) \leq (z(C_R)+1) O(d/k)\mathcal{H}_{d+1}(C).
    \]
    So necessarily, $z(C_R) \geq \Omega(k/d)$. We conclude that there is $w = k-O(1)$ that guarantees that $z(C^w)\leq -1$ and $\vol_{d+1}(C\setminus C^w)\leq O(d/k) \cdot  \vol_{d+1}(C)$. The conclusion follows.
\end{proof}

%%------------------------------------%%
%%------------------------------------%%

\subsection{Main result for the case $n=1$}
We can now present our main result for $n=1$. 
\begin{theorem}\label{thm:Main-OneDimension}
    There exists a point $\bar{\mathbf{x}}=(\bar{z},\bar{x})\in S$ such that every halfspace $H$ containing $\bar{\mathbf{x}}$, satisfies
    \[
    \vol_d(S\cap H) \geq \left( \frac{1}{e} - O\left(\frac{d}{k}\right)\right) \vol_d (S).
    \]
    In particular, there exists a universal constant $\alpha>0$ such that if $k\geq \alpha d$, then
    \[
    \mathcal{F}(S) \geq \frac{1}{2}\left(\frac{d}{d+1}\right)^{d}.
    \]
\end{theorem}

\begin{proof}
    Let $C'\subseteq C$ to be the set guaranteed by \Cref{lem:smaller_set_integral_centroid}, and take $\bar{\mathbf{x}}$ to be the centroid of $C'$. Let $H$ be a halfspace that contains $\bar{\mathbf{x}}$. We have that 
    \begin{equation}\label{eq:InProof-GrunbaumToC'}
        \vol_{d+1}(C\cap H)\geq \vol_{d+1}(C'\cap H) \geq \frac{1}{e}\vol_{d+1}(C')\geq \left(\frac{1}{e}-O(d/k)
    \right) \cdot \vol_{d+1}(C).
    \end{equation}
    Using the upper bound from \Cref{lem:upper_and_lower_bound_n_1}, we get that
    \begin{align}
        \vol_{d+1}(C\cap H) \geq \left(\frac{1}{e} - O(d/k) \right) \cdot \vol_{d}(S).\label{eq:almost_conclusion_thm_n1}
    \end{align}

    If the length of $C\cap H$ along the first dimension exceeds $k/2$, we can apply the lower bound from \Cref{lem:upper_and_lower_bound_n_1} on $C\cap H$ and $S\cap H$ (instead of $C$ and $S$), i.e.,
    \begin{align*}
    \vol_{d}(S\cap H) \geq \left(1-O(d/k) \right) \cdot \vol_{d+1}(C\cap H), 
    \end{align*}
    and conclude replacing this in \Cref{eq:almost_conclusion_thm_n1}. 
    Otherwise, the length of $C\setminus H$ along the first dimension exceeds $k/2$, and therefore, we can apply the upper bound from \Cref{lem:upper_and_lower_bound_n_1} on $C\setminus H$ and $S\setminus H$, i.e.,
    \[
    \vol_d(S\setminus H) \leq \left(1+O(d/k)\right) \vol_{d+1}(C\setminus H).
    \]
    Then, together with the original lower bound on $C$ and $S$ from \Cref{lem:upper_and_lower_bound_n_1} and using \eqref{eq:InProof-GrunbaumToC'} to write $\vol_{d+1}(C) \leq e(1+O(d/k))\vol_{d+1}(C\cap H)$, we get that 
    \begin{align*}
    \vol_d(S\cap H) &= \vol_d(S) -  \vol_d(S\setminus H)\\
    &\geq (1-O(d/k))\cdot \vol_{d+1}(C) - (1+O(d/k))\cdot \vol_{d+1}(C\setminus H)\\
    &= (1-O(d/k))\cdot \vol_{d+1}(C\cap H) - O(d/k)\vol_{d+1}(C\setminus H)\\
    &\geq (1-O(d/k))\cdot \vol_{d+1}(C\cap H) - O(d/k)\vol_{d+1}(C)\\
    &\geq (1-O(d/k))\cdot \vol_{d+1}(C\cap H) - eO(d/k)(1+O(d/k))\vol_{d+1}(C\cap H)\\
    &= (1-O(d/k))\cdot \vol_{d+1}(C\cap H). 
    \end{align*}
    So we can conclude by replacing this in \Cref{eq:almost_conclusion_thm_n1}.
\end{proof}
%%====================================%%
%%====================================%%

\section{The general case}\label{sec:general}
Throughout this section, we consider the general case $n\geq 1$, that is, $C\subseteq \R^{n+d}$ and $S=C\cap (\Z^n\times \R^d)$.  We first need to define what ``large set'' means in this context. For $n=1$, we said the set $C$ was large if the length of the segment $\proj_{\R}(C)$ was large. Analogously, we will consider $C$ to be large in $\R^{n+d}$ if its projection $\proj_{\R^n}(C)$ contains a ball of large radius. Assuming without loss of generality that $0\in C$, and denoting by $\mathbb{B}_n$ the Euclidean unit ball in $\R^n$, let $k$ be such that
\[
k\mathbb{B}_n \subset \proj_{\R^n}(C).
\]
Our goal is to prove that there exists a universal constant $\alpha>0$ such that
\begin{equation}
k\geq \alpha d^2n^{3/2} \implies \mathcal{F}(S) \geq \frac{1}{2^n}\left(\frac{d}{d+1}\right)^d.
\end{equation}

 Having already proved our bound for the case $n=1$, a naive strategy would be to apply the same ideas inductively over the dimensions of $\R^n$. However, such an approach would yield a bound that grows exponentially on $n$, which we want to avoid. Thus, we depart from the idea of conic approximations but maintain the following key observation: ``away from the boundary'' of $C$, any point $z\in \proj_{\R^n}(C)$ satisfies
\begin{align}
\vol_d(C\cap (\{z\}\times \R^d)) \approx \vol_{n+d}(C\cap (\{w\ :\ \|z-w\|_{\infty}\leq 1/2\}\times \R^{d}))
\label{eq:approx_volume_slice}
\end{align}

In what follows, we will consider the following definitions, illustrated in \Cref{fig:Boxes}:
\begin{itemize}
    \item For $z\in\R^n$, we denote the \textbf{slice} of $C$ induced by $z$ as $S_{z}(C) = C\cap (\{z\}\times \R^d)$ (solid blue line inside the set, see \Cref{fig:Boxes}).
    \item For $z\in\R^n$, we denote the $n$-dimensional \textbf{box} of $z$ as $\mathrm{Box}_z^n = \{w\in \R^n\, :\, \|w-z\|_{\infty}\leq 1/2\}$ (solid red square over the grid, see \Cref{fig:Boxes}).
    \item We denote the \textbf{rectangular cut} of $C$ induced by $z$ as $B_z(C) = C\cap (\mathrm{Box}_{z}^n\times\R^d)$ (volumetric section of the convex body in blue, see \Cref{fig:Boxes}).
\end{itemize}

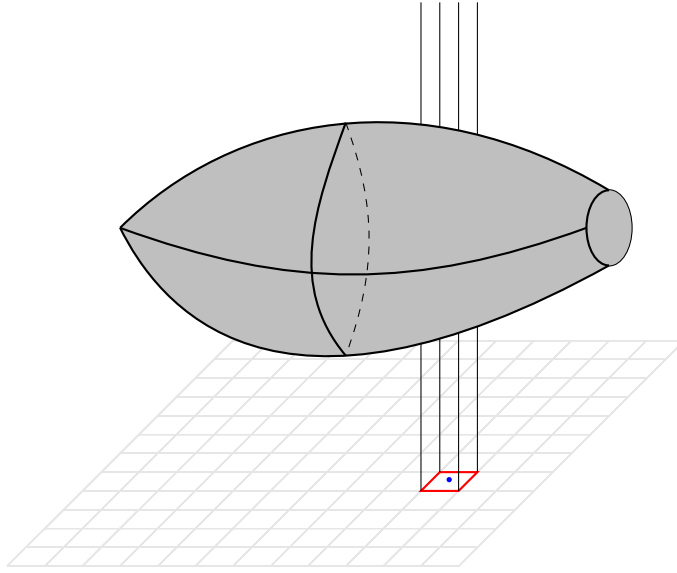
\begin{figure}[t]
    \centering
    \begin{tikzpicture}[scale = 0.8]
		
		%draw a grid in the x-y plane
		\foreach \x in {3,3.5,...,9}
		\foreach \y in {-4.5,-4.25,...,-1.5}
		{
			\draw[very thin,gray!20!white] (\x-4.5,-4.5) -- (\x-1.5,-1.5);
			\draw[very thin,gray!20!white] (3+\y,\y) -- (9+\y,\y);
		}

        \draw[red,thick] (7.5-3.5, -3.5) -- (8-3.5, -3.5) -- (8-3.25, -3.25) -- (7.5-3.25, -3.25) -- cycle;

        \fill[blue] (7.725-3.35,-3.35) circle (1pt);

        \draw[blue, very thick] (7.725-3.35,-1.25) -- (7.725-3.35,1);
        \fill[blue] (7.725-3.35,-1.25) circle (1pt);
        \fill[blue] (7.725-3.35,1) circle (1pt);
        
		\draw (7.5-3.5, -3.5)-- (7.5-3.5,-1.2);
        \draw[blue,dashed] (7.5-3.5,-1.2) -- (7.5-3.5,0.75);
        \draw (7.5-3.5,0.75) -- (7.5-3.5, 2.2);
        
        \draw (8-3.5, -3.5) -- (8-3.5,-1);
        \draw[blue, dashed] (8-3.5,-1) -- (8-3.5,0.6);
        \draw (8-3.5,0.6) --(8-3.5, 2.2);
        
        \draw (7.5-3.25, -3.25) -- (7.5-3.25,-1.4);
        \draw[blue, dashed] (7.5-3.25,-1.4)-- (7.5-3.25,1.2);
        \draw (7.5-3.25,1.2)--(7.5-3.25, 2.2);
        
        \draw (8-3.25, -3.25) -- (8-3.25,-1.2);
        \draw[blue, dashed] (8-3.25,-1.2) -- (8-3.25,1.1);
        \draw (8-3.25,1.1)--(8-3.25, 2.2);

        %\fill[blue] (7.5-3.5,-1.2) circle (1pt);
        %\fill[blue] (7.5-3.25,-1.4) circle (1pt);
        %\fill[blue] (8-3.5,-1) circle (1pt);
        %\fill[blue] (8-3.25,-1.2) circle (1pt);

        %\fill[blue] (7.5-3.5,0.75) circle (1pt);
        %\fill[blue] (7.5-3.25,1.2) circle (1pt);
        %\fill[blue] (8-3.5,0.6) circle (1pt);
        %\fill[blue] (8-3.25,1.1) circle (1pt);
        
        \filldraw[blue,fill opacity=0.2] (7.5-3.5,-1.2) to[out=270, in=160] (7.5-3.25,-1.4) -- (8-3.25,-1.2) to[out=170, in=280]  (8-3.5,-1) -- (7.5-3.5,-1.2);

        \draw[blue] (7.5-3.5,0.75) to[out=80, in=230] (7.5-3.25,1.2) -- (8-3.25,1.1) to[out=240, in=80]  (8-3.5,0.6) -- (7.5-3.5,0.75);

        \fill[blue, opacity=0.2] (7.5-3.5,0.75) to[out=80, in=230] (7.5-3.25,1.2) -- (8-3.25,1.1) 
        %to[out=240, in=80]  (8-3.5,0.6)
        -- (8-3.25,-1.2) to[out=170, in=280]  (8-3.5,-1) -- (7.5-3.5,-1.2);
        
    % Dibuja el conjunto convexo no simétrico con bordes lisos
    \fill[name path=convexo, fill=gray!50!white, fill opacity=0.1] (0,0) 
        .. controls (1.5,1.5) and (4,2) .. (6.5,0.5)  % Borde superior
        -- (6.5,-0.5) % Extremo derecho inferior
        .. controls (2,-3) and (0.5,-1) .. (0,0); % Borde inferior

    \draw[thick] (0,0) .. controls (1.5,1.5) and (4,2) .. (6.5,0.5);
    \draw[thick] (6.5,-0.5) .. controls (2,-3) and (0.5,-1) .. (0,0);
    \draw[thick] (6.5,0.5) arc[start angle=90, end angle=450, x radius=0.3, y radius=0.5];
    \fill[fill=gray!50!white, opacity=0.1](6.5,-0.5) arc[start angle=270, end angle=450, x radius=0.3, y radius=0.5];

    % Línea desde (0,0) a (6.2,0) con perspectiva 3D
    \draw[thick] (0,0) to[out=340, in=200] (6.2,0);
    \draw[thick] (3,1.4) to[out=250, in=130] (3,-1.7);
    \draw[dashed] (3,1.4) to[out=290, in=70] (3,-1.7);
\end{tikzpicture}
    \caption{Illustration of slices, boxes and rectangular cut with $n=2$ and $d=1$.}
    \label{fig:Boxes}
\end{figure}
With this notation, \Cref{eq:approx_volume_slice} can be written as 
\[
\vol_d(S_z(C))\approx \vol_{n+d}(B_z(C)).
\] 
Our strategy can be split in the following steps:

\begin{enumerate}

        \item \textbf{Approximate volume with slices.}
    The volume of a convex body that contains a large ball in the projection is well approximated by the measure of its slices. More precisely, 
    if $\proj_{\R^n}(C)$ contains a ball of radius $\Omega(k)$, then $\vol_{n+d}(C) = \vol_d(C\cap (\Z^n\times \R^d))\cdot (1\pm O(d n^{3/4}/k^{1/2}))$ (see \Cref{lem:approx_mu_with_nu}). We prove this in two steps.
    
    \begin{enumerate}
        \item Take the set $C_{-\varepsilon}= (1-\varepsilon)\cdot  C$. We have that $\vol_{n+d}(C_{-\varepsilon})=(1-\varepsilon)^{n+d} \vol_{n+d}(C)$, and for every point $z\in \proj_{\R^n}(C_{-\varepsilon})$ the ball of radius $\varepsilon r$ centered at $z$ is contained in $\proj_{\R^n}(C)$ (a simple consequence of \Cref{lem:thales_corollary}).
        
        \item If a point in the projection has a large ball around, then the volume of the corresponding $d$-dimensional slice approximates well the volume of the $(n+d)$-dimensional rectangular cut. More precisely, for a point $z\in \proj_{\R^n}(C)$, if a ball of radius $r$ centered at $z$ is contained in $\proj_{\R^n}(C)$, then $\vol_{n+d}(B_z(D)) = \vol_d(S_z(D)) \cdot (1\pm O(d\sqrt{n}/r))$. See \Cref{lem:slice_with_ball_around}.
    \end{enumerate}
    
    \item \textbf{Move the centroid.} We can assume the centroid of $C$ is in $\Z^n\times \R^d$ after losing only a small fraction of the volume. More precisely, we can shift $C$ and obtain a new set $C'$ with centroid in $\Z^n\times \R^d$, so that $\vol_{n+d}(C\triangle C') \leq  \vol_{n+d}(C)\cdot O(\sqrt{n}(n+d)/k)$.  See \Cref{lem:shift_centroid}.
    
    \item \textbf{Cut and use the continuous bound.} If we cut $C'$ with an arbitrary hyperplane passing through its centroid and denote the two sides by $A'$ and $B'$, then both $\vol_{n+d}(A')$ and $\vol_{n+d}(B')$ are at least $\vol_{n+d}(C')/e$.
    
    \item \textbf{Approximate volumes of sides by the slices.} Either $\proj_{\R^n}(A')$ or $\proj_{\R^n}(B')$ must contain a ball of radius $k/2$. In general, if the union of two convex sets contains a ball of radius $k$, one of them contains a ball of radius $k/2$~\cite{Kadets2005Covering}. Assume without loss of generality that $A'$ contains the ball. The volumes of $C'$ and $A'$ are well approximated by the slices, so the volume of $B'$ is also well approximated by the slices.
\end{enumerate}
%%====================================%%
%%====================================%%

\subsection{Technical Lemmas for the general case}

In this section, we write precise formulas in the statements so that at the end we obtain an explicit bound on $k$. Since this makes the computations heavier, we also include, before each proof, a high-level \emph{sketch of proof} using asymptotic notation. 

For a convex set $D$, we have that for any pair of elements $z,w\in D$,  $u = \frac{1}{1+\varepsilon} z + \frac{\varepsilon}{1+\varepsilon} w \in D$, and then $z+\varepsilon w = (1+\varepsilon)u \in (1+\varepsilon) D$, so we can write the following lemma (see \Cref{fig:thales}).
\begin{lemma}
    \label{lem:thales_corollary}
    Let $D\subseteq \R^n$ be a convex set. For all $z,w \in D$, and $\varepsilon>0$, we have that $z+\varepsilon w \in (1+\varepsilon)\cdot D$.
\end{lemma}

\begin{figure}
    \centering
    \begin{tikzpicture}
%epsilon=0.2
\filldraw[thick, fill = gray!20!white, fill opacity = 0.5] (0,0) ellipse (4cm and 2cm);
\draw[dashed, opacity = 0.5] (0,0) ellipse (4.8cm and 2.4cm);

\fill[blue] (0,0) circle (1pt) node[below left]{$0$};
\fill[blue] (2,1.5) circle (1pt) node[left, xshift=-0.2cm]{$z$};
\fill[blue] (3.9,0.3) circle (1pt) node[below right]{$w$};
\fill[blue] (2.78,1.56) circle (1pt) node[right]{$z+\varepsilon w$};
\fill[blue] (2.32,1.3) circle (1pt)
node[below]{$u$};
\draw[blue] (0,0) -- (2.78,1.56);
\draw[blue] (2,1.5) -- (3.9,0.3);
\draw[blue] (2,1.5) -- (0,0);
\draw[blue] (2.78,1.56) -- (3.9,0.3);
\draw[blue] (0,0) -- (3.9,0.3);
\draw[blue] (2,1.5) -- (2.78,1.56);

\node at (-3,1) {$D$};
\node at (-4,2) {$(1+\varepsilon)D$};
\end{tikzpicture}
    \caption{A graphic representation of \Cref{lem:thales_corollary}}
    \label{fig:thales}
\end{figure}
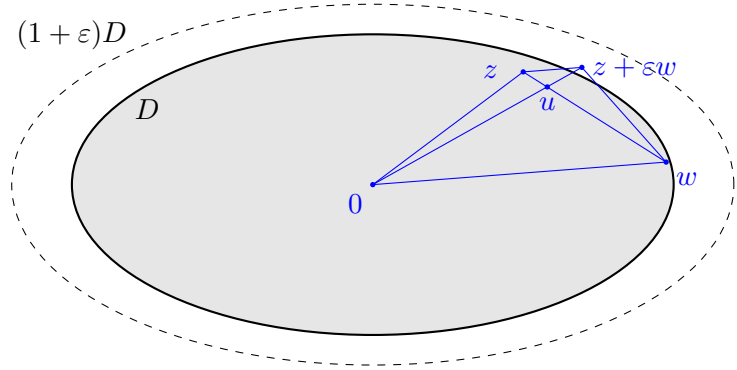

\begin{lemma}
    \label{lem:slice_with_ball_around}
    For a convex body $C\subseteq \R^{n+d}$ and a point $z\in \proj_{\R^n}(C)$ such that $\proj_{\R^n}(C)$ contains a ball of radius $r>\sqrt{n}/2$ around $z$, we have that
    \[
    \vol_d \left(S_z(C)\right) \cdot \left(1-\frac{\sqrt{n}}{2r} \right)^d\leq \vol_{n+d} (B_z(C)) \leq \vol_d (S_z(C))\cdot \left(1+\frac{\sqrt{n}}{2r} \right)^d.
    \]
\end{lemma}

\begin{proof}
    Since $B(z,r)\subset\proj_{\R^n}(C)$, we have that the box $\textup{Box}^n_z = \{ w\in \R^n: \|w-z\|_\infty \leq 1/2\}\subset \proj_{\R^n}(C)$. Then, we have the following formula:
    \[
    \vol_{n+d} (B_z(C)) = \int_{w\in \textup{Box}_z^n} \vol_d (S_w(C)) \, dw,
    \]
    where $dw$ is the Lebesgue measure in $\R^n$. Because of this formula, and since  $\vol_n (\textup{Box}^n_z)=1$, it is enough to bound $\vol_d (S_w(C))$ in terms of $\vol_d (S_z(C))$.

    Take a given $w\in \textup{Box}_z^n$. We have that $w$ is at a distance at most $\sqrt{n}/2$ of $z$. Consider a point $\mathbf{y}\in C$ such that $\proj_{\R^n}(\mathbf{y})-z$ is in the ray $\R_+(w-z) := \{ t(w-z)\mid t\geq 0 \}$, and $\|\proj_{\R^n}(\mathbf{y})-z\|_2 = r$ ($\mathbf{y}$ is guaranteed to exist due to the inclusion $B(z,r)\subset \proj_{\R^n}(C)$). By convexity, the cone  $K_z=\cone(\mathbf{y},S_z(C))$ is contained in $C$. Therefore, since $\|w-z\|_2\leq r$, the intersection of $K_z$ and the affine subspace $\{w\}\times \R^d$ is contained in $S_w(C)$ (see \Cref{fig:slices}). 
    
    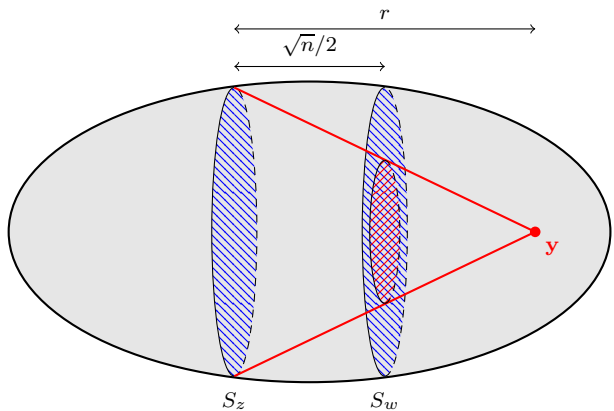
\begin{figure}[ht]
    \centering
    \begin{tikzpicture}

\filldraw[thick, fill = gray!20!white, fill opacity = 0.5] (0,0) ellipse (4cm and 2cm);

% S_z
\draw[pattern=north west lines, pattern color = blue, fill opacity = 0.3] (-1,1.92) arc[start angle=90, end angle=270, x radius=0.3, y radius=1.92];
\draw[dashed,pattern=north west lines, pattern color = blue, fill opacity = 0.3] (-1,-1.92) arc[start angle=270, end angle=450, x radius=0.3, y radius=1.92];

\draw[pattern=north west lines, pattern color = blue, fill opacity = 0.3] (1,1.92) arc[start angle=90, end angle=270, x radius=0.3, y radius=1.92];
\draw[dashed,pattern=north west lines, pattern color = blue, fill opacity = 0.3] (1,-1.92) arc[start angle=270, end angle=450, x radius=0.3, y radius=1.92];

\fill[red] (3,0) circle (2pt) node[below right]{\scriptsize $\mathbf{y}$};
\draw[red,thick] (3,0) -- (-1,1.92);
\draw[red,thick] (3,0) -- (-1,-1.92);
\draw[pattern=north east lines, pattern color = red, fill opacity = 0.3] (1,0.95) arc[start angle=90, end angle=270, x radius=0.2, y radius=0.95];
\draw[dashed,pattern=north east lines, pattern color = red, fill opacity = 0.3] (1,-0.95) arc[start angle=270, end angle=450, x radius=0.2, y radius=0.95];

%\fill[red] (-3,0) circle (2pt) node[below left]{\scriptsize $\mathbf{y}'$};
%\draw[red,thick] (-3,0) -- (1,1.92);
%\draw[red,thick] (-3,0) -- (1,-1.92);
%\draw[pattern=north east lines, pattern color = red, fill opacity = 0.3] (-1,0.95) arc[start angle=90, end angle=270, x radius=0.2, y radius=0.95];
%\draw[dashed,pattern=north east lines, pattern color = red, fill opacity = 0.3] (-1,-0.95) arc[start angle=270, end angle=450, x radius=0.2, y radius=0.95];

\draw[<->] (-1,2.2)--(1,2.2) node[above,midway]{\scriptsize $\sqrt{n}/2$};

\draw[<->] (-1,2.7)--(3,2.7) node[above,midway]{\scriptsize $r$};

\node[below] at (-1,-2){\scriptsize $S_z$};
\node[below] at (1,-2){\scriptsize $S_w$};

\end{tikzpicture}
    \caption{A graphic representation of the procedure used in \Cref{lem:slice_with_ball_around} to compare the volume of two slices.}
    \label{fig:slices}
\end{figure}
    
    The cone $K_z$ has height $r$, and $S_w$ is at distance at most $\sqrt{n}/2$ from the base $S_z(C)$. So, its intersection with the affine subspace defined by $w$ is a rescaled translate of $S_z(C)$, with a rescaling factor no smaller than $(r-\sqrt{n}/2)/r$. Therefore,
    
    %the height $h$ of the cone $\cone(\mathbf{y},S_w(C))$ is larger than $r-\sqrt{n}/2$ and so
    \begin{align*}
    \vol_d (S_z(C))\cdot \left(\frac{r-\sqrt{n}/2}{r}\right)^d 
    &\leq \vol_d (S_w(C)).
    \end{align*}
    
    Similarly, there must exist a point $\mathbf{y}'\in C$ such that $\proj_{\R^n}(\mathbf{y}')-z$ points in the exact opposite direction, and $\|\proj_n(\mathbf{y}')-z\|_2 =r$. The cone $K' = \cone(\mathbf{y}',S_{w}(C))$ is contained in $C$ by convexity, so its intersection with the affine subspace $\{z\}\times\R^d$ is contained in $S_z$. Therefore, reasoning as above,
    \[
    \vol_d (S_w(C))\cdot \left(\frac{r}{r+\sqrt{n}/2}\right)^{d} \leq \vol_d (S_z(C)).
    \]
    Putting these bounds back in the integral formula and noting that $\int_{\textup{Box}_z^n}1dw = 1$, we conclude the proof of the lemma.
\end{proof}

\begin{lemma}
    \label{lem:approx_mu_with_nu}
    Let $C\subseteq \R^{n+d}$ be a convex body such that $\proj_{\R^n}(C)$ contains a ball of radius $k$, and let $S = C\cap(\Z^n\times\R^d)$. If $5\frac{d n^{3/4}}{k^{1/2}}\leq 1$, we have that
    \[
    \vol_{n+d} (C)\cdot \left(1 - 5\frac{d n^{3/4}}{k^{1/2}}\right) \leq  \vol_d (S) \leq \vol_{n+d} (C) \cdot \left(1 + 5\frac{d n^{3/4}}{k^{1/2}}\right).
    \]
\end{lemma}

\begin{proof}
    Let us assume first that the ball of radius $k$ contained in $C$ is centered at $0$.  Let $\varepsilon=1/(n^{1/4}k^{1/2})$. Note that $\varepsilon\leq 1/5$. Define $C_{-\varepsilon}= C \cdot (1-\varepsilon)$. By \Cref{lem:thales_corollary}, for every $z\in \proj_{\R^n}(C_{-\varepsilon})$, the ball of radius $k\varepsilon(1-\varepsilon)$ centered at $z$ is contained in $(1+\varepsilon)\proj_{\R^n}(C_{-\varepsilon})\subset \proj_{\R^n}(C)$. Denote by $\textup{Int}_n(\cdot)=\Z^n\cap \proj_{\R^n}(\cdot)$. By \Cref{lem:slice_with_ball_around}, for every $z\in \textup{Int}_n(C_{-\varepsilon})$, we can approximate $\vol_d (S_z(C))$ with $\vol_{n+d}(B_z(C))$.
    Therefore, we have that
    \begin{align*}
        \vol_d (S) 
        &= \sum_{z\in \textup{Int}_n(C)} \vol_d (S_z(C))\\
        &\geq \sum_{z\in \textup{Int}_n(C_{-\varepsilon})} \vol_d (S_z(C))\\
        & \geq \left(1-\frac{\sqrt{n}}{2k\varepsilon(1-\varepsilon)} \right)^d\sum_{z\in \textup{Int}_n(C_{-\varepsilon})}  \vol_{n+d}(B_z(C))\\
        &=\left(1-\frac{\sqrt{n}}{2k\varepsilon(1-\varepsilon)} \right)^d \cdot \vol_{n+d} \left( \bigcup_{z\in \textup{Int}_n(C_{-\varepsilon})}  B_z(C) \right).
    \end{align*}
    We would like to conclude the lower bound here by replacing $\bigcup_{z\in \textup{Int}_n(C_{-\varepsilon})}  B_z(C)$ with $C_{-\varepsilon}$. However, there might be points in $C_{-\varepsilon}$ that are not covered by any box $B_z$. Let $\mathbf{y}$ be such a point. In that case, $\proj_{\R^n}(\mathbf{y})$ is at distance at most $\sqrt{n}/2$ of a point $w\in \Z^n\setminus \proj_{\R^n}(C_{-\varepsilon})$.
    We claim that $\mathbf{y}\notin C_{-2\varepsilon}$. Indeed, if this were not the case, \Cref{lem:thales_corollary} yields that the ball centered in $\proj_{\R^n}(\mathbf{y})$ with radius $k\varepsilon(1-2\varepsilon)$ is contained in $(1+\varepsilon)\proj_{\R^n}(C_{-2\varepsilon})\subset \proj_{\R^n}(C_{-\varepsilon})$. Noting that
    \[
    k\varepsilon(1-2\varepsilon) \geq \frac{3k}{5k^{1/2}n^{1/4}}\geq \frac{3k}{5k^{1/2}n^{1/4}}\cdot\frac{5dn^{3/4}}{k^{1/2}} = 3\sqrt{n} \geq \|\proj_{\R^n}(\mathbf{y})-w\|,
    \]
    we would deduce that $w\in \proj_{\R^n}(C_{-\varepsilon})$, which is a contradiction. This proves the claim. 
    We deduce that $(1-2\varepsilon)C \subset \bigcup_{z\in \textup{Int}_n(C_{-\varepsilon})}  B_z(C)$. Recall that for any $a,b\geq 0$, $(1-a)(1-b)\geq 1-a-b$ and, if $a\in [0,1]$, then $(1-a)^b\geq 1-ab$ also holds. Then we can write
    \begin{align*}
        \vol_d (S) 
        &\geq
        \left(1-\frac{\sqrt{n}}{2k\varepsilon(1-\varepsilon)} \right)^d \cdot \vol_{n+d} (C_{-2\varepsilon})\\
        & = \left(1-\frac{\sqrt{n}}{2k\varepsilon(1-\varepsilon)} \right)^d \cdot \left( 
        1-2\varepsilon\right)^{n+d} \cdot 
        \vol_{n+d} (C)\\
        & \geq \left(1-\frac{5\sqrt{n}}{8k\varepsilon} -2(n+d)\varepsilon\right) \cdot 
        \vol_{n+d} (C)\\
        &= \left( 1- \frac{5d n^{3/4}}{8k^{1/2}}- 2\frac{(n+d)}{n^{1/4}k^{1/2}} \right) \cdot \vol_{n+d} (C)\\
        &\geq \left( 1- \frac{5d n^{3/4}}{8k^{1/2}}- 2\frac{dn}{n^{1/4}k^{1/2}} \right) \cdot \vol_{n+d} (C) \geq  \left( 1- 3\frac{d n^{3/4}}{k^{1/2}}\right) \cdot \vol_{n+d}(C), 
    \end{align*}
    
     For the upper bound we define $C_{+\varepsilon}= (1+\varepsilon)\cdot C$. We have that
    \begin{align*}
        \vol_d (S) 
        &= \sum_{z\in \textup{Int}_n(C)} \vol_d (S_z(C))\\
        &\leq \sum_{z\in \textup{Int}_n(C)} \vol_d (S_z(C_{+\varepsilon})).
    \end{align*}
    The inequality holds because if the origin is contained in $C$, then $C\subseteq C_{+\varepsilon}$, and therefore, for all $z\in \Z^n$, $S_z(C)\subseteq S_z(C_{+\varepsilon})$. Now, for every $z\in \textup{Int}_n(C)$, the ball of radius $k\varepsilon$ centered at $z$ is contained in $
    C_{+\varepsilon}$, so we can apply \Cref{lem:slice_with_ball_around}.  Consider the following facts:
    \begin{enumerate}
        \item $(n+d)\varepsilon \leq nd\varepsilon = \frac{dn^{3/4}}{k^{1/2}}\leq \frac{1}{5}\leq \frac{1}{2}$.
        \item $\frac{d\sqrt{n}}{2k\varepsilon} = \frac{dn^{3/4}}{2k^{1/2}} \leq \frac{1}{10}\leq \frac{1}{2}$.
        \item For any $a,b>0$, if $a\leq 1/2$, then $\frac{1}{1-a}\leq (1+2a)$.
        \item For any $a,b>0$, if $ab\leq 1$, then $(1+a)^b\leq 1+2ab$. This is because $(1+a)^b\leq e^{ab}$, and by the convexity of $e^x$, we have that $e^x\leq 1+(e-1)x$ for all $x\in [0,1]$.
    \end{enumerate}
   Then we can write,
    \begin{align*}
        \vol_d (S) 
        &\leq \frac{1}{\left(1-\frac{\sqrt{n}}{2k\varepsilon}\right)^d} 
        \sum_{z\in \textup{Int}_n(
        C)} \vol_{n+d} (B_z(C_{+\varepsilon}))\\
        &= \frac{1}{\left(1-\frac{\sqrt{n}}{2k\varepsilon}\right)^d} \cdot \vol_{n+d}\left( 
        \bigcup_{z\in \textup{Int}_n(C)} B_z(C_{+\varepsilon}) \right)\\
        &\leq \frac{1}{\left(1-\frac{\sqrt{n}}{2k\varepsilon}\right)^d} \cdot \vol_{n+d} (C_{+\varepsilon})\\
        &= \frac{(1+\varepsilon)^{n+d}}{\left(1-\frac{\sqrt{n}}{2k\varepsilon}\right)^d} \cdot \vol_{n+d} (C)\\
        &\leq \left(\frac{1}{\left(1-\frac{d\sqrt{n}}{2k\varepsilon}\right)} + \frac{2(n+d)\varepsilon}{\left(1-\frac{d\sqrt{n}}{2k\varepsilon}\right)}\right)\vol_{n+d} (C)\\
        &\leq \left( 1 + 2\frac{d\sqrt{n}}{2k\varepsilon} + \frac{20}{9}(n+d)\varepsilon\right)\vol_{n+d} (C)
        \leq \left(1+ 4\frac{dn^{3/4}}{k^{1/2}}\right) \cdot \vol_{n+d} (C).
    \end{align*}
    We deduce, under the assumption that the ball of radius $k$ is centered at $0$, that
    \[
    \vol_{n+d} (C)\cdot \left(1 - 4\frac{d n^{3/4}}{k^{1/2}}\right) \leq  \vol_d (S) \leq \vol_{n+d} (C) \cdot \left(1 + 4\frac{d n^{3/4}}{k^{1/2}}\right).
    \]
    For the general case, we can translate $C$ so $0$ is the closest mixed-integer point to the center of the ball of radius $k$. After this translation, we can replicate the development considering a radius $k' = k-\frac{\sqrt{n}}{2}$. Then, using that $n^{3/2}/k \leq 1/25$, we have that
    \begin{align*}
        \frac{1}{(k')^{1/2}} &= \frac{1}{k^{1/2}}\sqrt{\frac{k}{k-\sqrt{n}/2}}\\
        &= \frac{1}{k^{1/2}}\sqrt{\frac{1}{1-n^{3/2}/(2nk)}}\leq \frac{1}{k^{1/2}}\sqrt{\frac{1}{1-1/50}} = \frac{5\sqrt{2}}{7}\frac{1}{k^{1/2}}\leq \frac{5}{4}\frac{1}{k^{1/2}}.
    \end{align*}
    Thus, the result follows.
\end{proof}

\begin{lemma}
    \label{lem:shift_centroid}
    Let $C\subseteq \R^{n+d}$ be a convex body such that $\proj_{\R^n}(C)$ contains a ball of radius $k>n$. There exists a vector $\mathbf{x}$ such that the centroid of $C+\mathbf{x}$ lies in $\Z^n\times \R^d$ and satisfies
    \[
    \vol_{n+d} \Big((C+\mathbf{x})\setminus C\Big)\leq \left(\left( 1+\frac{\sqrt{n}}{2k}\right)^{n+d} -1 \right)\cdot \vol_{n+d}(C).
    \]
\end{lemma}

\begin{proof}
    Let $\bar{\mathbf{x}} = (\bar{z},\bar{x})$ such that $B(\bar{z},k)\subset \proj_{\R^n}(C)$. Let $(z,x)$ be the centroid of $C$. There must be a point $z'\in \Z^n$ at distance at most $\sqrt{n}/2$ from $z$. Thus, $z''=\bar{z}+\frac{2k}{\sqrt{n}}(z'-z)$ is at a distance at most $k$ from $\bar{z}$, and therefore lies in $\proj_{\R^n}(C)$. That means there is $x''\in \R^d$ such that $\mathbf{x}''= (z'', x'')\in C$. We shift $C$ by $\frac{\sqrt{n}}{2k} (\mathbf{x}''-\bar{\mathbf{x}})$, so the projection of the centroid of the shifted set is $z'$.
    
    Let $D= C-\bar{\mathbf{x}}$. Since $0$ and $\mathbf{x}''-\bar{\mathbf{x}}$ are in $D$, \Cref{lem:thales_corollary} implies that 
     \[D+\frac{\sqrt{n}}{2k} (\mathbf{x}''-\bar{\mathbf{x}}) \subseteq \left(1+\frac{\sqrt{n}}{2k}\right)\cdot D,
     \] 
     and therefore, $(D+\frac{\sqrt{n}}{2k} (\mathbf{x}''-\bar{\mathbf{x}}))\setminus D \subseteq \left((1+\frac{\sqrt{n}}{2k})\cdot D \right)\setminus D $. The result follows by noting that  $\vol_{n+d}\left(\big(C+\frac{\sqrt{n}}{2k} (\mathbf{x}''-\bar{\mathbf{x}})\big)\setminus C\right) = \vol_{n+d}\left(\big(D+\frac{\sqrt{n}}{2k} (\mathbf{x}''-\bar{\mathbf{x}})\big)\setminus D\right)$.
\end{proof}

%%====================================%%
%%====================================%%

\subsection{Main result for the general case}
We now present our main result.
\begin{theorem}
    \label{thm:general_n}
    Let $C\subseteq \R^{n+d}$ be a convex body such that $\proj_{\R^n}(C)$ contains a ball of radius $k$. There is a point $\mathbf{x}\in S=C\cap (\Z^n\times \R^d)$ such that for every halfspace $H$ that contains $\mathbf{x}$,
    \[
    \vol_d(S\cap H ) \geq \left( \frac{1}{e} - 32\frac{dn^{3/4}}{k^{1/2}}\right) \cdot \vol_d (S).
    \]
    In particular, there exists a universal constant $\alpha>0$ such that if $k\geq \alpha d^2n^{3/2}$, then
    \[
    \mathcal{F}(S) \geq \frac{1}{2^n}\left(\frac{d}{d+1}\right)^{d}.
    \]
\end{theorem}

\begin{proof}
    Without loss of generality, let us suppose that 
    \begin{equation}\label{eq:InProof-RequiredGap}
     32\frac{dn^{3/4}}{k^{1/2}} \leq \frac{1}{e},   
    \end{equation}
    since otherwise the result is trivial. Note that in this case the inequality  $5\frac{dn^{3/4}}{k^{1/2}} \leq 1$ required in \Cref{lem:approx_mu_with_nu} also holds. Let $C'=C+\mathbf{x}$, where $\mathbf{x}$ is the vector guaranteed to exist by \Cref{lem:shift_centroid} so that the centroid of $C'$ is in $S' = C'\cap (\Z^n\times \R^d)$. Let $H$ be an arbitrary affine halfspace defined by an affine hyperplane that passes through the centroid of $C'$. Let $A'=  C'\cap H$, $A=C\cap H$, $B' = C'\setminus H$, and $B=C\setminus H$. Without loss of generality, we assume that all these sets are full-dimensional. We have that $\vol_{n+d}(A')\geq (1/e)\cdot \vol_{n+d} (C')$ and $\vol_{n+d}(B')\geq (1/e)\cdot \vol_{n+d}(C')$. Either $\proj_{\R^n}(A)$ or $\proj_{\R^n}(B)$ contains a ball of radius $k/2$. Since we will prove the same lower bound for both sides, we assume that $\proj_{\R^n}(A)$ contains a ball of radius $k/2$. For notational convenience, let $\delta = 5\frac{dn^{3/4}}{k^{1/2}}$.
    
    Since $\delta \leq 1$ required in \Cref{lem:approx_mu_with_nu} also holds directly from \eqref{eq:InProof-RequiredGap} and noting that $(1-a)(1-b)\geq 1-a-b$ for every $a,b\geq 0$,  we have that
    \begin{align*}
    \vol_d(A\cap (\Z^n\times \R^d)) 
    &\geq \vol_{n+d}(A) \cdot \left(1-\sqrt{2}\delta\right)\\
    &\geq \left(\vol_{n+d}(A')-\vol_{n+d} \left( C'\setminus C \right)\right)
    \cdot \left(1-\sqrt{2}\delta\right)\\
    &\geq \vol_{n+d} (C)\left(\frac{1}{e} - \frac{\sqrt{n}(n+d)}{k}\right)\cdot 
    \left(1-\sqrt{2}\delta\right)\\
    &\geq \frac{1}{e}\vol_{n+d} (C)\left( 1 - \frac{e}{32}\cdot d \frac{n^{3/4}}{k^{1/2}}\right)\cdot 
    \left(1-5\sqrt{2}\frac{dn^{3/4}}{k^{1/2}}\right)\\
    &\geq \vol_{n+d} (C) \cdot \frac{1}{e} \cdot 
    \left(1-8\frac{dn^{3/4}}{k^{1/2}}\right),
    \end{align*}  
    %\begin{align*}
    %\vol_d(A\cap (\Z^n\times \R^d)) 
    %&\geq \vol_{n+d}(A) \cdot \left(1-5\sqrt{2}\frac{dn^{3/4}}{k^{1/2}}\right)\\
    %&\geq \left(\vol_{n+d}(A')-\vol_{n+d} \left( C'\setminus C \right)\right)
    %\cdot \left(1-5\sqrt{2}\frac{dn^{3/4}}{k^{1/2}}\right)\\
    %&\geq \vol_{n+d} (C)\left(\frac{1}{e} - \frac{\sqrt{n}(n+d)}{k}\right)\cdot 
    %\left(1-5\sqrt{2}\frac{dn^{3/4}}{k^{1/2}}\right)\\
    %&\geq \frac{1}{e}\vol_{n+d} (C)\left( 1 - \frac{e}{13}\cdot d \frac{n^{3/4}}{k^{1/2}}\right)\cdot 
    %\left(1-5\sqrt{2}\frac{dn^{3/4}}{k^{1/2}}\right)\\
    %&\geq \vol_{n+d} (C) \cdot \frac{1}{e} \cdot 
    %\left(1-8\frac{dn^{3/4}}{k^{1/2}}\right),
    %\end{align*}
    where in the third inequality we applied \Cref{lem:shift_centroid}, and in the fourth inequality we used the following relation that relies on \eqref{eq:InProof-RequiredGap}
    
    \[
    \frac{\sqrt{n}(n+d)}{k} \leq \frac{\sqrt{n}(2dn)}{k} = \left(2\frac{n^{3/4}}{k^{1/2}}\right)\cdot d \frac{n^{3/4}}{k^{1/2}}\leq \left(\frac{2}{32ed}\right)\cdot d \frac{n^{3/4}}{k^{1/2}}\leq \frac{1}{32}\cdot d \frac{n^{3/4}}{k^{1/2}}.
    \] 
    
    On the other hand,
    
    \begin{align*}
        \vol_d (B\cap (\Z^n\times \R^d)) 
        &= \vol_d (S) - 
        \vol_d(A \cap (\Z^n\times \R^d)) \\
        &\geq \vol_{n+d} (C)\left(1-\delta\right) - \left(1+\sqrt{2}\delta\right)\vol_{n+d} (A)\\
        &= \left( \vol_{n+d} (C) - \vol_{n+d} (A) \right)
        \cdot \left(1+\sqrt{2}\delta\right) - (1+\sqrt{2})\delta \vol_{n+d} (C)\\
        &= \vol_{n+d}(B) \cdot \left(1+\sqrt{2}\delta\right) - (1+\sqrt{2})\delta \vol_{n+d} (C)\\
        &\geq \left( \vol_{n+d}(B') - \vol_{n+d} (C'\setminus C) \right) \left(1+\sqrt{2}\delta\right) - (1+\sqrt{2})\delta \vol_{n+d} (C)\\
        & \geq \left( \frac{1}{e}\vol_{n+d}(C') - \vol_{n+d} (C'\setminus C) \right) \left(1+\sqrt{2}\delta\right) - (1+\sqrt{2})\delta \vol_{n+d} (C)
        \\
        &\geq  \frac{1}{e}\vol_{n+d} (C)\left( 1 - \frac{e}{5\cdot 32}\cdot \delta \right)\cdot 
    \left(1+\sqrt{2}\delta\right) - (1+\sqrt{2})\delta \vol_{n+d} (C)\\
    &= \frac{1}{e}\vol_{n+d} (C)\left( 1 - \delta \cdot \left( \frac{e}{5\cdot 32} - \sqrt{2} + \frac{e\sqrt{2}}{5\cdot 32} \delta +e(1 +\sqrt{2})\right)\right)\\
        &\geq \vol_{n+d}(C) \cdot \frac{1}{e} \cdot 
    \left(1-\delta\cdot\left(e(1+\sqrt{2})+\frac{e}{5\cdot 32} +\frac{e\sqrt{2}}{32} - \sqrt{2}\right)\right)\\
    &\geq \vol_{n+d}(C) \cdot \frac{1}{e} \cdot 
    \left(1-27\frac{dn^{3/4}}{k^{1/2}}\right), 
    \end{align*}
    
%where the last inequality follows by noting that $5(1+\sqrt{2})e\leq 33$ and that $\vol_{n+d} (A)\leq \vol_{n+d} (C)$. 
Now, directly by \Cref{lem:approx_mu_with_nu},
    \begin{align*}
        \vol_d (S) \leq \vol_{n+d} (C) \cdot \left(1+5\frac{dn^{3/4}}{k^{1/2}}\right).
    \end{align*}
    Therefore,
    \begin{align*}
         \frac{\min\{\vol_d (A\cap (\Z^n\times\R^d)), \vol_d (B\cap (\Z^n\times\R^d)) \}}{\vol_d (S)}
        &\geq \frac{1}{e} \cdot \frac{1-27\frac{dn^{3/4}}{k^{1/2}}}{1+5\frac{dn^{3/4}}{k^{1/2}}}\\
        &\geq \frac{1}{e} \cdot \frac{\left(1-27\frac{dn^{3/4}}{k^{1/2}}\right)\left(1-5\frac{dn^{3/4}}{k^{1/2}}\right)}{\left(1+5\frac{dn^{3/4}}{k^{1/2}}\right)\left(1-5\frac{dn^{3/4}}{k^{1/2}}\right)}.\\
        &\geq \frac{1}{e} \cdot \left(1-32\frac{dn^{3/4}}{k^{1/2}}\right).
    \end{align*}
    
    The second part of the proof is analogous to the final step in the proof of \Cref{thm:Main-OneDimension}. Indeed, if $k\geq \alpha d^2n^{3/2}$, we get that
    \[
    \mathcal{F}(S) \geq \frac{1}{e}\left(1 - \frac{32}{\sqrt{\alpha}}\right).
    \]
    We can choose $\alpha = \left(\frac{128}{4-e}\right)^2$ so $\frac{1}{e}\left(1 - \frac{32}{\sqrt{\alpha}}\right) = \frac{1}{4} \geq \frac{1}{2^n}\left(\frac{d}{d+1}\right)^d$.
\end{proof}

\begin{remark}\label{rem:unimodular-ball}
    As we will discuss in the next section, unimodular linear transformations preserve mixed-integer volumes. Thus, \Cref{thm:general_n} also holds if $\proj_{\R^n}(C)$ contains a unimodular copy of a ball of radius $k$.
\end{remark}

%%====================================%%
%%====================================%%

\subsection{A new Threshold for the Lattice Width}

Here we discuss the relation between the condition in our result, namely that the projection contains a ball of large radius, and the large-lattice-width condition in the result of Basu and Oertel. We show that if the lattice width of a convex set is large, then, after an appropriate transformation, its projection in $\R^n$ contains a large ball. Consequently, we can reinterpret our bound in terms of the lattice width.

Here, $\textup{Flt}(n)$ is the flatness constant in dimension $n$, which is defined as the supremum of the lattice width of convex sets contained in $\R^n\setminus \Z^n$. It is known that $\textup{Flt}(n)\leq n^{5/2}$ ~\cite[Chapter VII, Theorem 8.3]{Barvinok2002Course}. Moreover, it has been shown that  $\textup{Flt}(n)\leq O(n^{3/2})$ \cite{Banaszczyk1999}.

\begin{proposition}
\label{prop:equivalence_lattice_width}
    Given a convex body $K\subseteq \R^{n}$ there is an unimodular linear transformation $L:\R^{n}\rightarrow \R^{n}$ such that $L(K)$ contains a ball of radius $\omega(K)/(2n^2\textup{Flt}(n))$.
\end{proposition}

\begin{proof}
    Let $\Delta_n$ be the $n$-dimensional standard simplex in $\R^n$, that is, $\Delta_n= \textup{conv}\{\mathbf{0},\mathbf{e}_1,\dots,\mathbf{e}_n\}$. Averkov, Hofscheier and Nill~\cite[Corollary~2.2]{Averkov2023Generalized} showed that a convex set $K\subseteq \R^n$ always contains a unimodular copy of $\frac{\omega(K)}{n\textup{Flt}(n)}\cdot \Delta_n$, i.e., there is an affine unimodular transformation (i.e. a unimodular transformation plus an integer translation) such that the image of $K$ contains $\frac{\omega(K)}{n\textup{Flt}(n)}\cdot \Delta_n$. 
    It is a well-known fact that the standard simplex $\Delta_n$ contains the hypercube $[0,1/n]^n$, which in turn contains a ball of radius $1/(2n)$, and therefore, the image of $K$ contains a ball of radius $\frac{\omega(K)}{2n^2\textup{Flt}(n)}$. By inverting the integer translation if necessary, the result follows.
\end{proof}

A unimodular linear transformation in $\R^n$ is a linear transformation that is invertible over the integers, i.e., it identifies points in $\Z^n$ with points in $\Z^n$. If we extend $L$, the linear transformation of the proposition, to $\R^{n+d}$ using the identity in $\R^d$, we obtain an invertible linear transformation $\bar{L}= (L,\text{Id})$ that preserves all the volumes in $\Z^n\times \R^d$. This means that we can effectively work with $\bar{L}(C\cap (\Z^n\times \R^d))$ instead of $C\cap (\Z^n\times \R^d)$, and any conclusion about the volume also holds for $C\cap (\Z^n\times \R^d$). Thus, we obtain the conclusion announced in \Cref{rem:unimodular-ball}, as well as the following corollary as a consequence of our theorem and \cite{Banaszczyk1999}.

\begin{corollary}\label{cor:Main-LatticeWidth}
     Let $C\subseteq \R^{n+d}$ be a convex body. There is a point $\mathbf{x}\in S=C\cap(\Z^n\times \R^d)$ such that for every halfspace $H$ that contains $\mathbf{x}$,
     \[ 
     \vol_d(H\cap S)
     \geq \left( \frac{1}{e} - 27\sqrt{2}\frac{dn^{7/4} \sqrt{\text{Flt(n)}}}{\sqrt{\omega(\proj_{\R^n}(C))}}\right)
     \cdot
     \vol_d (S).
     \]   
     In particular, there is a universal constant $\alpha>0$ such that if $\omega(\proj_{\R^n}(C)) \geq \alpha d^2 n^5$, then
    \[
    \mathcal{F}(S) \geq \frac{1}{2^n} \left( \frac{d}{d+1}\right)^d.
    \]
\end{corollary}

%%====================================%%
%%====================================%%

\paragraph{Acknowledgments.}  Both authors were partially funded by the Center for Mathematical Modeling (CMM), FB210005, BASAL funds
for centers of excellence (ANID-Chile). The second author was partially funded by the project FONDECYT 1251159 (ANID-Chile).
The second author wishes to thank A. Basu for his lectures and useful discussions at IPCO 2023, which motivated this research.  Part of this research was conducted during the first author's postdoctoral fellowship at the Center for Mathematical Modeling (CMM). Part of this research was conducted during a 3-month visit of the second author to the IMT, Université Toulouse Paul Sabatier, CNRS, INSA Toulouse, UT, INU, UT Capitole, UT2J, funded by CNRS through the \textit{Poste Rouge} program. We thank the anonymous referees for their thorough reviews, which allowed us to significantly improve our initial manuscript and refine some results. 

\section*{Declarations}
\paragraph{Conflict of interest.} The authors have no relevant financial or non-financial interests to disclose.
%\newpage
%\nocite{*}
\bibliographystyle{abbrvnat}
 \bibliography{lit}
\end{document}